\documentclass{article}

\usepackage[margin=1in]{geometry}
\usepackage{amsmath}
\usepackage{amssymb}
\usepackage{amsthm}
\usepackage{multirow}
\usepackage{enumerate}
\usepackage{color}
\usepackage{longtable}
\usepackage{caption2}

\usepackage{float}

\floatstyle{plaintop}
\restylefloat{table}

\allowdisplaybreaks[1]
\newcommand{\ra}[1]{\renewcommand{\arraystretch}{#1}}

\newtheorem{theorem}{Theorem}

\newtheorem{lemma}[theorem]{Lemma}
\newtheorem{proposition}[theorem]{Proposition}
\newtheorem{definition}[theorem]{Definition}
\newtheorem{example}[theorem]{Example}
\newtheorem{result}[theorem]{Result}

\numberwithin{theorem}{section}
\numberwithin{equation}{section}
\numberwithin{table}{section}

\newtheorem{remark}[theorem]{Remark}

\newcommand{\bb}{\mathbb}
\newcommand{\Z}{\bb{Z}}

\newcommand{\cL}{\mathcal{L}}

\newcommand{\de}{\delta}

\newcommand{\la}{\lambda}

\newcommand{\pr}{\prime}

\newcommand{\sm}{\setminus}
\newcommand{\es}{\emptyset}

\newcommand{\ol}{\overline}
\newcommand{\lan}{\langle}
\newcommand{\ran}{\rangle}

\newcommand{\lf}{\lfloor}
\newcommand{\rf}{\rfloor}
\newcommand{\lc}{\lceil}
\newcommand{\rc}{\rceil}

\newcommand{\F}{\mathbb{F}}
\newcommand{\Fp}{\mathbb{F}_p}

\newcommand{\Fq}{\mathbb{F}_q}

\newcommand{\PGaL}{\text{P$\Gamma$L}}
\newcommand{\Spec}{\text{Spec}}

\long\def\symbolfootnote[#1]#2{\begingroup%
\def\thefootnote{\fnsymbol{footnote}}\footnote[#1]{#2}\endgroup}

\begin{document}

\title{Intersection distribution, non-hitting index and Kakeya sets in affine planes}
\author{Shuxing Li \and Alexander Pott}
\date{}
\maketitle

\symbolfootnote[0]{
S.~Li and A.~Pott are with the Faculty of Mathematics, Otto von Guericke University Magdeburg, 39106 Magdeburg, Germany (e-mail: shuxing\_li@sfu.ca, alexander.pott@ovgu.de).
}

\begin{abstract}
In this paper, we propose the concepts of intersection distribution and non-hitting index, which can be viewed from two related perspectives. The first one concerns a point set $S$ of size $q+1$ in the classical projective plane $PG(2,q)$, where the intersection distribution of $S$ indicates the intersection pattern between $S$ and the lines in $PG(2,q)$. The second one relates to a polynomial $f$ over a finite field $\Fq$, where the intersection distribution of $f$ records an overall distribution property of a collection of polynomials $\{f(x)+cx \mid c \in \Fq\}$. These two perspectives are closely related, in the sense that each polynomial produces a $(q+1)$-set in a canonical way and conversely, each $(q+1)$-set with certain property has a polynomial representation. Indeed, the intersection distribution provides a new angle to distinguish polynomials over finite fields, based on the geometric property of the corresponding $(q+1)$-sets. Among the intersection distribution, we identify a particularly interesting quantity named non-hitting index. For a point set $S$, its non-hitting index counts the number of lines in $PG(2,q)$ which do not hit $S$. For a polynomial $f$ over a finite field $\Fq$, its non-hitting index gives the summation of the sizes of $q$ value sets $\{f(x)+cx \mid x \in \Fq\}$, where $c \in \Fq$. We derive lower and upper bounds on the non-hitting index and show that the non-hitting index contains much information about the corresponding set and the polynomial. More precisely, using a geometric approach, we show that the non-hitting index is sufficient to characterize the corresponding point set and the polynomial, when it is very close to the lower and upper bounds. Moreover, we employ an algebraic approach to derive the non-hitting index and the intersection distribution of several families of point sets and polynomials. As an application, we consider the determination of the sizes of Kakeya sets in affine planes. The polynomial viewpoint of intersection distributions enable us to compute the size of a few families of Kakeya sets with nice algebraic properties. Finally, we describe the connection between these new concepts and various known results developed in different contexts and propose some future research problems.


\end{abstract}

\section{Intersection distribution: two sides of a coin}\label{sec-1}

This paper proposes the concept of intersection distribution and investigates this concept from two related perspectives, namely a geometric viewpoint concerning point sets in the classical projective plane $PG(2,q)$, and an algebraic viewpoint concerning polynomials over finite fields. In this section, we introduce a series of definitions related to polynomials over finite fields and point sets in $PG(2,q)$. Then, we explain how these definitions can be connected from the viewpoint of the intersection distribution. These connections, together with the application to the Kakeya sets in affine planes, justify our motivation to study the intersection distribution.

First of all, we supply a list of definitions related to a polynomial $f$ over a finite field, after which some illustrative remarks are in order.

\begin{definition}\label{def-poly}
Let $f$ be a polynomial over $\Fq$.
\begin{itemize}
\item[(1)] For $0 \le i \le q$, define
$$
v_i(f)=|\{(a,b) \in \Fq^2 \mid \mbox{$f(x)-ax-b=0$ has $i$ solutions in $\Fq$} \}|.
$$
The sequence $(v_i(f))_{i=0}^{q}$ is the intersection distribution of $f$. The integer $v_0(f)$ is the non-hitting index of $f$.
\item[(2)] For $c \in \Fq$ and $0 \le i \le q$, define $M_i(f,c)$ to be the number of elements in $\Fq$, which occur $i$ times in the multiset $\{f(x)-cx \mid x \in \Fq\}$. The sequence $(M_i(f,c))_{i=0}^q$ is the multiplicity distribution of $f$ at $c$.
\item[(3)] For $c \in \Fq$, define $V_{f,c}=\{f(x)+cx \mid x \in \Fq\}$, which is the value set of the polynomial $f(x)+cx$. Define $N_f=\{ c \in \Fq \mid |V_{f,c}|=q \}$, namely, the set of elements $c \in \Fq$ such that $f(x)+cx$ is a permutation polynomial.
\item[(4)] Let $q$ be an even prime power. The polynomial $f$ is an o-polynomial, if $f$ is a permutation polynomial and $f(x)+cx$ is $2$-to-$1$ for each $c \in \Fq^*$.
\end{itemize}
\end{definition}

\begin{remark}\label{rem-polyint}
\quad
\begin{itemize}
\item[(1)] There are $q^2+q$ lines in the affine plane $AG(2,q)$. The intersection distribution of $f$ records the intersection pattern between the graph $\{(x,f(x)) \mid x \in \Fq\}$ and $q^2$ lines in $\Fq^2$, excluding the $q$ vertical lines $\{x=c \mid c \in \Fq\}$. We omit $q$ vertical lines to make the definition of intersection distribution more compact, and in fact, since each vertical line intersects the graph of $f$ in exactly one point, no information is lost by restricting to non-vertical lines. By definition, for each linear function $l$, the polynomials $f$ and $f+l$ have the same intersection distribution. Moreover, the non-hitting index of $f$ counts the number of $q^2$ non-vertical lines in $\Fq^2$ which do not hit the graph $\{(x,f(x)) \mid x \in \Fq \}$.
\item[(2)] The sequence $(M_i(f,c))_{i=0}^q$ records the multiplicities of elements in the value multiset $\{f(x)-cx \mid x \in \Fq\}$, where an element not in the multiset has multiplicity $0$. For a polynomial $f$ over $\Fq$ and $0 \le i \le q$, we have
    $$
    v_i(f)=\sum_{c \in \Fq}M_i(f,c).
    $$
    Hence, to compute the intersection distribution of $f$, it suffices to compute its multiplicity distribution at each element $c \in \Fq$. In particular,
    $$
    v_0(f)=\sum_{c \in \Fq}M_0(f,c)=\sum_{c \in \Fq} (q-|V_{f,c}|)=q^2-\sum_{c \in \Fq} |V_{f,c}|.
    $$
    Therefore, $v_0(f)$ records the summation of the sizes of the $q$ value sets $V_{f,c}$, where $c \in \Fq$.
\item[(3)] The o-polynomial only exists when $q$ is even. When $q$ is odd, as we shall see, $f(x)=x^2$ plays an analogous role as o-polynomial. For a detailed account of o-polynomials, please refer to {\rm \cite[Section 6]{CM}}.
\end{itemize}
\end{remark}

Now, we introduce a second viewpoint of intersection distribution related to point sets in $PG(2,q)$. Later, we shall see that the polynomial and the point set viewpoints are closely related. A point set of $PG(2,q)$ of size $k$ is called a $k$-set of $PG(2,q)$. Next, we supply a list of definitions related to point sets in $PG(2,q)$ and some remarks.

\begin{definition}\label{def-set}
Let $D$ be a point set in $PG(2,q)$ with $|D|>1$.
\begin{itemize}
\item[(1)] For $0 \le i \le q+1$, define $u_i(D)$ to be the number of lines in $PG(2,q)$, which intersect $D$ in exactly $i$ points. The sequence $(u_i(D))_{i=0}^{q+1}$ is the intersection distribution of $D$. The integer $u_0(D)$ is the non-hitting index of $D$. The largest integer $2 \le n \le q+1$, such that $u_n(D)>0$, is the degree of $D$.
\item[(2)] A point $P \in D$ is called an internal nucleus of $D$ if each line through $P$ meets $D$ in at most one more point. For a $(q+1)$-set $S$, a point $P \notin S$ is a nucleus to $S$, if each line through $P$ intersects $S$ in exactly one point.
\end{itemize}
\end{definition}

\begin{remark}
\quad
\begin{itemize}
\item[(1)] The intersection distribution of $D$ reflects the intersection pattern between the set $D$ and lines in $PG(2,q)$. The non-hitting index of $D$ counts the number of lines which do not hit $D$. Indeed, some configurations with special intersection distribution have been intensively studied, including:
\begin{itemize}
\item[(a)] If $D$ is a $k$-set of degree $n$ in $PG(2,q)$, then $D$ is called a $(k;n)$-arc in the literature {\rm\cite[Chapter 12]{Hirs}}. In particular, if $n=2$, then $D$ is called a $k$-arc, which is one of the most well-studied configurations in classical projective planes {\rm\cite[Chapters 9,10]{Hirs}}.
\item[(b)] If $u_0(D)=u_{q+1}(D)=0$, namely, the set $D$ meets every line and does not contain any, then $D$ is a blocking set in $PG(2,q)$ {\rm\cite[Chapter 13]{Hirs}}.
\end{itemize}
\item[(2)] A subset $D$ contains an internal nucleus only if $|D| \le q+2$.
\end{itemize}
\end{remark}

Having defined the intersection distribution of polynomials and point sets, the following result, which is a combination of Lemma~\ref{lem-polyrep} and Proposition~\ref{prop-intdispoly}, establishes the relation between them. We use $\lan (x,y,z) \ran$ to denote a point in $PG(2,q)$.

\begin{result}\label{res-polyset}
\begin{itemize}
\item[(1)] Each polynomial $f$ over $\Fq$ corresponds to a $(q+1)$-set
\begin{equation*}
S_f=\{ \lan (x,f(x),1) \ran \mid x \in \Fq \} \cup \{ \lan (0,1,0) \ran \}
\end{equation*}
in $PG(2,q)$, where $\lan (0,1,0) \ran$ is an internal nucleus of $S_f$.
\item[(2)] Let $S$ be a $(q+1)$-set in $PG(2,q)$ containing an internal nucleus. Then there exists a polynomial $f$ over $\Fq$, such that $$
    S:=S_f=\{\lan (x,f(x),1) \ran \mid x \in \Fq\} \cup \{ \lan (0,1,0) \ran \},
    $$
    where $\lan (0,1,0) \ran$ is an internal nucleus of $S_f$.
\item[(3)] The intersection distribution of $f$ implies that of $S_f$ and vice versa. In particular, $v_0(f)=u_0(S_f)$.
\end{itemize}
\end{result}

Before we proceed to illustrate the significance of the intersection distribution and the non-hitting index, we first introduce a new type of equivalence between polynomials, which relies on the $(q+1)$-sets derived from them.

\begin{definition}[projective equivalence]\label{def-projequiv}
Two polynomials $f$ and $f^{\pr}$ are projectively equivalent, if the two $(q+1)$-sets $S_f$ and $S_{f^{\pr}}$ are isomorphic, in the sense that there exists an automorphism of $PG(2,q)$, which is an element of $\PGaL(3,q)$, mapping $S_f$ into $S_{f^{\pr}}$.
\end{definition}

By Result~\ref{res-polyset}(3), the intersection distribution of a polynomial is an invariant under the projective equivalence. The following result, which follows from Propositions~\ref{prop-upperbound} and \ref{prop-intdispoly}, greatly motivates our investigation.

\begin{result}\label{res-nonhitting}
Let $S$ be a $(q+1)$-set in $PG(2,q)$ and $f$ a polynomial over $\Fq$. Then we have the following.
\begin{itemize}
\item[(1)] $0 \le u_0(S) \le\frac{q(q-1)}{2}$. Moreover, $u_0(S)=0$ if and only if $S$ is a line and $u_0(S)=\frac{q(q-1)}{2}$ if and only if $S$ is a $(q+1)$-arc.
\item[(2)] $v_0(f)=u_0(S_f) \ge q-1$. Moreover, $v_0(f)=q-1$ if and only if $f$ is a linear function.
\item[(3)] $v_0(f)=u_0(S_f) \le \frac{q(q-1)}{2}$. Moreover, $v_0(f)=u_0(S_f)=\frac{q(q-1)}{2}$ if and only if $S_f$ is a $(q+1)$-arc and one of the following holds.
    \begin{itemize}
      \item[(3a)] When $q$ is even, for exactly one element $c \in \Fq$, we have $f(x)-cx$ being an o-polynomial.
      \item[(3b)] When $q$ is odd, $f$ is projectively equivalent to $x^2$.
    \end{itemize}
\end{itemize}
\end{result}

\begin{remark}
Let $S$ be a $(q+1)$-set in $PG(2,q)$ and $f$ a polynomial over $\Fq$.
\begin{itemize}
\item[(1)] Recall that the degree of $S$ is an integer between $2$ and $q+1$. The non-hitting index $u_0(S)$ achieves the lower bound $0$ if and only if the point set $S$ forms a line, which has the maximum degree $q+1$. The non-hitting index $u_0(S)$ achieves the upper bound $\frac{q(q-1)}{2}$ if and only if the point set $S$ forms a $(q+1)$-arc, which has the minimum degree $2$. Therefore, the non-hitting index of a $(q+1)$-set indicates its distance to the aforementioned two extremal configurations.
\item[(2)] The non-hitting index $v_0(f)$ attains the minimal value $q-1$ if and only if $f$ is a linear function. The non-hitting index $v_0(f)$ attains the maximal value $\frac{q(q-1)}{2}$ if and only if $f$ is an o-polynomial up to adding a linear function when $q$ is even or $f$ is projectively equivalent to $x^2$ when $q$ is odd. Therefore, the non-hitting index of a polynomial indicates its distance to the aforementioned polynomials.
\item[(3)] We note that a Kakeya set in the affine plane $AG(2,q)$ can be viewed in a dual way, as a $(q+2)$-set in $PG(2,q)$ with an internal nucleus. We are greatly inspired by some recent work characterizing Kakeya sets with small and large sizes {\rm\cite{BDMS,BM,DM1}}, which have essentially characterized $(q+2)$-sets with an internal nucleus by their non-hitting index. A more detailed account is in Section~\ref{sec-4}.
\end{itemize}
\end{remark}

Now we explain how Result~\ref{res-nonhitting} inspired our investigation and describe the organization of the remaining sections. First of all, by definition, a $(q+1)$-arc is a $(q+1)$-set of degree $2$, that is, $u_i(D)=0$ for $3 \le i \le q+1$. By knowing this, we can actually determine the intersection distribution $(u_i(D))_{i=0}^{q+1}$ (see Proposition~\ref{prop-upperbound}). Instead of specifying the whole intersection distribution, Result~\ref{res-nonhitting}(3) characterizes a $(q+1)$-arc $S$ by its non-hitting index $u_0(S)=\frac{q(q-1)}{2}$, which achieves the upper bound.
Having $(q+1)$-arcs, which is one of the most interesting configurations in $PG(2,q)$, as a primary example, one may ask if there are any other $(q+1)$-sets which can also be characterized by their non-hitting index. This constitutes the main theme of Section~\ref{sec-2}.

Second, the non-hitting index $v_0(f)$ of a polynomial $f$ measures the distance from $f$ to the o-polynomial when $q$ is even, or to $x^2$ when $q$ is odd, and also its distance to a linear function. Therefore, non-hitting index offers a new perspective to study polynomials over finite fields. The estimate and computation of the non-hitting index, or more generally, the intersection distribution of polynomials, is the main purpose of Section~\ref{sec-3}.

Third, in Section~\ref{sec-4}, we present an application of the intersection distribution, which concerns Kakeya sets in affine planes. While computing the size of Kakeya sets is difficult in general, the results concerning intersection distributions in Section~\ref{sec-3} immediately lead to several families of Kakeya sets with known sizes.

Finally, in Section~\ref{sec-5}, we mention some work related to the intersection distribution and the non-hitting index. By observing these connections, we hope that the techniques in several papers can be applied to study the intersection distribution and the non-hitting index. We conclude that the intersection distribution and the non-hitting index deserve further investigation, where several open problems are proposed. To keep the main text focus on the conceptual ideas, most technical proofs are intentionally presented in the Appendices~\ref{sec-appendixA} and \ref{sec-appendixB}.

\section{Characterizing $(q+1)$-sets by their non-hitting indices}\label{sec-2}

Let $D$ be a point set in the classical projective plane $PG(2,q)$. A natural question is, to what extent, the interaction between $D$ and the lines of $PG(2,q)$, implies information about $D$? There has been intensive research along this direction, in terms of arcs \cite[Chapters 9,10,12]{Hirs}, ovals and hyperovals \cite[Chapter 8]{Hirs}, blocking sets \cite[Chapter 13]{Hirs}, to name just a few.
In this section, we investigate a few classes of $(q+1)$-sets $S$, which can be fully characterized by their non-hitting index $u_0(S)$. For most results in this section, proofs are presented in Appendix~\ref{sec-appendixA}.

Now we introduce some notation and terminology which will be used later. Let $P$ be a point and $\ell$ a line in $PG(2,q)$. Then we write $P \in \ell$ if $P$ is on the line $\ell$ and $P \notin \ell$ otherwise. Given $k$ points $P_i$, $1 \le i \le k$, if they are collinear, then we use $\ol{P_1P_2\cdots P_k}$ to denote the line passing through them. Let $D$ be a point set of $PG(2,q)$. A line $\ell$ in $PG(2,q)$ is called an external line (resp. a tangent line) to $D$, if $\ell$ does not intersect $D$ (intersects $D$ in one point). For $2 \le i \le q+1$, a line $\ell$ in $PG(2,q)$ is called an $i$-secant line to $D$, if $\ell$ intersects $D$ in $i$ points. More generally, a line $\ell$ in $PG(2,q)$ is called a secant line to $D$, if $\ell$ intersects $D$ in at least two points. By abusing of notation, we sometimes also use $\ell$ to denote the point set consisting of the $q+1$ points on the line $\ell$.

From now on, we always use $S$ to denote a $(q+1)$-set in $PG(2,q)$. Note that for simplicity, whenever referring to the intersection distribution, we only list the $u_i$'s with nonzero values and omit all the $u_i$'s equal to zero. The following proposition gives an upper bound on $u_0(S)$ and characterizes the $(q+1)$-set achieving this upper bound.

\begin{proposition}\label{prop-upperbound}
Let $S$ be a $(q+1)$-set in $PG(2,q)$. The following equations hold.
\begin{align*}
\sum_{i=0}^{q+1} u_i(S)&=q^2+q+1, \\
\sum_{i=1}^{q+1} iu_i(S)&=(q+1)^2, \\
\sum_{i=2}^{q+1} i(i-1)u_i(S)&=q(q+1),
\end{align*}
which implies $u_0(S)=\frac{q(q-1)}{2}-\sum_{i=3}^{q+1} \frac{(i-1)(i-2)}{2} u_i(S)$. Consequently, $u_0(S) \le \frac{q(q-1)}{2}$, where the equality holds if and only if $S$ is a $(q+1)$-arc, whose intersection distribution is
$$
u_0(S)=\frac{q(q-1)}{2}, \; u_1(S)=q+1, \; u_2(S)=\frac{q(q+1)}{2}.
$$
\end{proposition}
\begin{proof}
It suffices to prove the three equalities, which follow immediately from \cite[Lemma 12.1]{Hirs}.
\end{proof}

Moreover, we are going to show that when $u_0(S)$ is close to the upper bound $\frac{q(q-1)}{2}$, the set $S$ can also be characterized by its non-hitting index $u_0(S)$. These $(q+1)$-sets are described in the following examples, which are very close to $(q+1)$-arcs.

\begin{example}\label{exam-evenlarge}
Let $q$ be an even prime power and $S$ a $(q+1)$-set, which is not a $(q+1)$-arc and contains a $q$-arc $A$, where $S=A \cup \{Q\}$. When $q$ is even, each $q$-arc is contained in a $(q+1)$-arc $T$ {\rm \cite[Corollary 10.13]{Hirs}}, say, $T=A \cup \{P\}$. Through each point of $T$, there exists exactly one tangent line to $T$ and these $q+1$ tangent lines intersect in one point $O$. By the definition of $S$, we have $Q \notin \{O,P\}$. By {\rm\cite[Corollary 8.8]{Hirs}}, the number of external lines, tangent lines and $2$-secant lines to $T$ through $Q$ are  $\frac{q}{2}$, $1$ and $\frac{q}{2}$. Suppose the unique tangent line to $T$ through $P$ is $\ell$. Then we have the following two cases.
\begin{itemize}
\item[(1)] If $q>2$ and $Q \notin \ell$, then the number of tangent lines, $2$-secant lines and $3$-secant lines to $S$ through $Q$ are $\frac{q}{2}$, $2$ and $\frac{q}{2}-1$. It is routine to verify that
    $$
      u_0(S)=\frac{q(q-1)}{2}-\frac{q}{2}+1, \; u_1(S)=\frac{5q}{2}-2, \; u_2(S)=\frac{q(q-2)}{2}+3, \; u_3(S)=\frac{q}{2}-1.
    $$

\item[(2)] If $Q \in \ell$, then the number of tangent lines and $3$-secant lines to $S$ through $Q$ are $\frac{q}{2}+1$ and $\frac{q}{2}$. It is routine to verify that
    $$
      u_0(S)=\frac{q(q-1)}{2}-\frac{q}{2}, \; u_1(S)=\frac{5q}{2}+1, \; u_2(S)=\frac{q(q-2)}{2}, \; u_3(S)=\frac{q}{2}.
    $$
\end{itemize}
Note that when $q=2$, we must have $Q \in \ell$, and therefore, (1) is vacuous in this case.
\end{example}

\begin{example}\label{exam-oddlarge}
Let $q$ be an odd prime power and $S$ a $(q+1)$-set which is not a $(q+1)$-arc and contains a $q$-arc $A$, where $S=A \cup \{Q\}$. When $q$ is odd, each $q$-arc is contained in a $(q+1)$-arc $T$ {\rm \cite[Theorem 10.28]{Hirs}}, say, $T=A \cup \{P\}$. Noting that $Q \notin T$, by {\rm\cite[Table 8.2]{Hirs}}, either there are two tangent lines to $T$ through $Q$, which implies the number of external lines, tangent lines and $2$-secant lines to $T$ through $Q$ are $\frac{q-1}{2}$, $2$ and $\frac{q-1}{2}$, or there is no tangent line to $T$ through $Q$, which implies the number of external lines and $2$-secant lines to $T$ through $Q$ are both $\frac{q+1}{2}$. Suppose the unique tangent line to $T$ through $P$ is $\ell$. Then we have the following two cases.
\begin{itemize}
\item[(1)] If $q>3$, and there are two tangent lines to $T$ through $Q$ and neither of them is $\ell$, then the number of tangent lines, $2$-secant lines and $3$-secant lines to $S$ through $Q$ are $\frac{q-1}{2}$, $3$ and $\frac{q-3}{2}$. It is routine to verify that
    $$
      u_0(S)=\frac{q(q-1)}{2}-\frac{q-3}{2}, \; u_1(S)=\frac{5q-7}{2}, \; u_2(S)=\frac{q^2-2q+9}{2}, \; u_3(S)=\frac{q-3}{2}.
    $$
\item[(2)] If there are two tangent lines to $T$ through $Q$ and one of them is $\ell$, or there is no tangent line to $T$ through $Q$, then the number of tangent lines, $2$-secant lines and $3$-secant lines to $S$ through $Q$ are $\frac{q+1}{2}$, $1$ and $\frac{q-1}{2}$. It is routine to verify that
    $$
      u_0(S)=\frac{q(q-1)}{2}-\frac{q-1}{2}, \; u_1(S)=\frac{5q-1}{2}, \; u_2(S)=\frac{q^2-2q+3}{2}, \; u_3(S)=\frac{q-1}{2}.
    $$
\end{itemize}
Note that when $q=3$, we must have $Q \in \ell$, and therefore, (1) is vacuous in this case.
\end{example}

In the sense of Proposition~\ref{prop-upperbound}, the non-hitting index of a $(q+1)$-set $S$ measures how much $S$ resembles a $(q+1)$-arc in $PG(2,q)$. For instance, Examples~\ref{exam-evenlarge} and \ref{exam-oddlarge} indicate that when $S$ is one step away from a $(q+1)$-arc, $u_0(S)$ is close to the upper bound $\frac{q(q-1)}{2}$. In the following, we aim to give a more precise description of this phenomenon. For this purpose, we need to introduce more concepts. Let $A$ be a subset of $S$, then $A$ is an $S$-maximal arc \cite[p. 309]{BB}, if
\begin{itemize}
\item[(a)] $A$ is an arc in $PG(2,q)$,
\item[(b)] For each point $P \in S \setminus A$, the set $A \cup \{P\}$ is not an arc in $PG(2,q)$.
\end{itemize}
We note that the $S$-maximal arc may not be unique. For instance, consider a $4$-set $S$ in $PG(2,3)$, where three points of $S$ are on one line $\ell$ and the remaining one is not on $\ell$. Then, $S$ has three distinct $S$-maximal $3$-arcs. Each $S$-maximal arc gives a best possible local approximation of $S$ by using points forming an arc. When $S$ is a $(q+1)$-set which contains an $S$-maximal $q$-arc, the intersection distribution of $S$ has been determined in Examples~\ref{exam-evenlarge} and \ref{exam-oddlarge}.

The second concept concerns the interaction between points of $S \sm A$ and $A$. A point $P \in S \setminus A$ is called a pro-arc point of $A$ if there exists exactly one $2$-secant line to $A$ through $P$. For each pro-arc point $P$ of $A$, the set $A \cup \{P\}$ is nearly an arc, since every line, except the $2$-secant line through $P$, intersects $A \cup \{P\}$ in at most two points. This justifies the name of pro-arc point. Let $\cL:=\cL(A)$ be the set of all $2$-secant lines to $A$. Let $S$ be a $(q+1)$-set and $A$ an $S$-maximal arc, define a set $B:=B(S,A)$ satisfying the following two conditions:
\begin{itemize}
\item[(a)] $B \subset \{ P \in S \sm A \mid \mbox{$P$ is a pro-arc point to $A$} \}$,
\item[(b)] for each $\ell \in \cL$ which contains a pro-arc point to $A$, we have $|B \cap \ell|=1$.
\end{itemize}
We note that the set $B$ may not be unique. For instance, when a line $\ell \in \cL$ contains two pro-arc points, then $B$ may contain either of the two. On the other hand, for different choices of $B$, their sizes remain the same. This is because, by definition, the size of $B$ is equal to the number of $2$-secant lines to $A$ which contain at least one pro-arc point. Given a set $B$ defined as above, the set $A \cup B$ is a largest possible subset of $S$ containing $A$, so that $u_3(A \cup B)>0$ and $u_i(A \cup B)=0$ for each $4 \le i \le q+1$. Thus, the set $B$ indicates a way of expanding $A$ to a largest possible subset $A \cup B$ of $S$, so that $A \cup B$ is still close to an arc, in the sense that each line meets $A \cup B$ in at most three points. Therefore, we call $B$ a pro-arc set with respect to $S$ and $A$.

Since the non-hitting index $u_0(S)$ indicates the similarity between $S$ and a $(q+1)$-arc, one may expect a connection between $u_0(S)$, the size of an $S$-maximal arc $A$, and the size of a pro-arc set $B$ with respect to $S$ and $A$. Indeed, employing the idea of \cite[Lemma 1.1]{BB} and \cite[Lemma 3.2]{BDMS}, we have the following upper bounds on $u_0(S)$ related to the sizes of $A$ and $B$. Again, we note that most technical proofs of the results in this section are included in Appendix~\ref{sec-appendixA}.

\begin{lemma}\label{lem-upperboundstructure}
Let $S$ be a $(q+1)$-set and $A$ an $S$-maximal $k$-arc. Let $B$ be a pro-arc set of size $l$ with respect to $S$ and $A$, with $0 \le l \le q+1-k$. Then we have the following.
\begin{itemize}
\item[(1)] Suppose for each $P \in S \sm A$, there are at most $\la$ tangent lines to $A$ through $P$, where $\la \le k$. Then $u_0(S) \le \frac{q(q-1)}{2}-(q+1-k)\frac{k-\la}{2}$.
\item[(2)]  $u_0(S) \le \frac{q(q-1)}{2}-2(q+1)+2k+l$.
\end{itemize}
\end{lemma}

In order to apply Lemma~\ref{lem-upperboundstructure}(1), one needs an estimate on the number of tangent lines to $A$ through each $P \in S \sm A$. To exploit Lemma~\ref{lem-upperboundstructure}(2), an estimate on the size of $B$ is required. The following lemma provides the bounds in need.

\begin{lemma}\label{lem-numtangent}
Let $S$ be a $(q+1)$-set and $A$ an $S$-maximal $k$-arc such that $k<q+1$. Let $P$ be a point of $S \setminus A$. Let $B$ be a pro-arc set of size $l$ with respect to $S$ and $A$. Then we have the following.
\begin{itemize}
\item[(1)] The number of tangent lines to $A$ through $P$ is at most $k-2$.
\item[(2)] If $q$ is even and $k > \frac{q}{2}+1$, the number of tangent lines to $A$ through $P$ is at most $q+2-k$.
\item[(3)] If $q$ is odd and $k > \frac{2q+4}{3}$, the number of tangent lines to $A$ through $P$ is at most $2(q+2-k)$.
\item[(4)] If through each point of $A$, there exists at most one $2$-secant line to $A$, which contains one point of $B$, then $l \le \lf \frac{k}{2} \rf$.
\end{itemize}
\end{lemma}

Combining Lemmas~\ref{lem-upperboundstructure} and \ref{lem-numtangent}, we have the following upper bounds on $u_0(S)$ which only involves the size of $S$-maximal $k$-arcs.

\begin{lemma}\label{lem-upperboundmaxarc}
Let $S$ be a $(q+1)$-set containing an $S$-maximal $k$-arc $A$ with $2 \le k \le q$. Then the following holds.
\begin{itemize}
\item[(1)] When $q$ is even, we have
\begin{equation*}
  u_0(S) \le \begin{cases}
    \frac{q(q-1)}{2}-(q+1-k) & \mbox{if $2 \le k < \frac{q}{2}+2$,} \\
    \frac{q(q-1)}{2}-(q+1-k)(k-\frac{q+2}{2}) & \mbox{if $\frac{q}{2}+2 \le k \le q$.}
  \end{cases}
\end{equation*}
\item[(2)] When $q$ is odd, we have
\begin{equation*}
  u_0(S) \le \begin{cases}
    \frac{q(q-1)}{2}-(q+1-k) & \mbox{if $2 \le k < \frac{2q+6}{3}$,} \\
    \frac{q(q-1)}{2}-\frac{3}{2}(q+1-k)(k-\frac{2q+4}{3}) & \mbox{if $\frac{2q+6}{3} \le k \le q$.}
  \end{cases}
\end{equation*}
\end{itemize}
\end{lemma}
\begin{proof}
Suppose that $q$ is even. Then by Lemma~\ref{lem-numtangent}(1)(2), for each $P \in S \sm A$, the number of tangent lines to $A$ through $P$ is at most $\min\{k-2,q+2-k\}$. By dividing into two cases $2 \le k < \frac{q}{2}+2$ and $\frac{q}{2}+2 \le k \le q$ and applying Lemma~\ref{lem-upperboundstructure}(1), we complete the proof of (1). The proof of (2) is analogous to that of (1).
\end{proof}

\begin{remark}
The upper bounds in Lemma~\ref{lem-upperboundmaxarc}(1)(2) consider the interaction between the $S$-maximal $k$-arc $A$ and the points of $S \sm A$. On the other hand, they do not take the internal relations among the points of $S \sm A$ into consideration. Hence, the two bounds are more favorable when the size $k$ of the $S$-maximal arc is close to $q$.
\end{remark}

Now we are ready to present our main theorems aiming to characterize $S$ when $u_0(S)$ is close to the upper bound $\frac{q(q-1)}{2}$. The following theorem concerns the $q$ even case, where we achieve a full description of the second and third largest value of $u_0(S)$. Given two integers $a \le b$, we use $[a,b]$ to denote the set of integers $\{x \in \Z \mid a \le x \le b\}$. In the case that $a>b$, we define $[a,b]$ to be the empty set.

\begin{theorem}\label{thm-evenupper}
Let $q$ be an even prime power. Let $S$ be a $(q+1)$-set which is not a $(q+1)$-arc. Then $u_0(S) \le \frac{q(q-1)}{2}-\frac{q}{2}+1$. Moreover, we have the following.
\begin{itemize}
\item[(1)] When $q \ge 16$, we have $u_0(S)=\frac{q(q-1)}{2}-\frac{q}{2}+1$ if and only if $S$ is of the form in Example~\ref{exam-evenlarge}(1).
\item[(2)] When $q \ge 16$, we have $u_0(S)=\frac{q(q-1)}{2}-\frac{q}{2}$ if and only if $S$ is of the form in Example~\ref{exam-evenlarge}(2).
\item[(3)] When $q \in \{2,4,8\}$, the $(q+1)$-sets $S$ satisfying $u_0(S) \in \{\frac{q(q-1)}{2}-\frac{q}{2}, \frac{q(q-1)}{2}-\frac{q}{2}+1\}$ are listed in Lemma~\ref{lem-evensmall}.
\end{itemize}
\end{theorem}

When $q$ is odd, the following theorem provides partial information about the second largest value of $u_0(S)$.

\begin{theorem}\label{thm-oddupper}
Let $q$ be an odd prime power. Let $S$ be a $(q+1)$-set and which is not a $(q+1)$-arc. Then we have the following.
\begin{equation}\label{eqn-oddupperbound}
   u_0(S) \le \begin{cases}
                  \frac{q(q-1)}{2}-\frac{q-3}{3} & \mbox{if $q \equiv 0 \bmod 3$ and $q \ge 9$,} \\
                  \frac{q(q-1)}{2}-\frac{q-1}{3} & \mbox{if $q \equiv 1 \bmod 3$ and $q \ge 13$,} \\
                  \frac{q(q-1)}{2}-\frac{q-2}{3} & \mbox{if $q \equiv 2 \bmod 3$ and $q \ge 11$,}
               \end{cases}
\end{equation}
where the above equality holds only if
$$
\begin{cases}
  \mbox{$S$ contains an $S$-maximal $\frac{2q+6}{3}$-arc} & \mbox{if $q \equiv 0 \bmod 3$ and $q \ge 9$,} \\
  \mbox{$S$ contains an $S$-maximal $\frac{2q+4}{3}$-arc} & \mbox{if $q \equiv 1 \bmod 3$ and $q \ge 13$,} \\
  \mbox{$S$ contains an $S$-maximal $\frac{2q+5}{3}$-arc} & \mbox{if $q \equiv 2 \bmod 3$ and $q \ge 11$.}
\end{cases}
$$
Moreover, when $q \in \{3,5,7\}$, the second largest value of $u_0(S)$ is described in Lemma~\ref{lem-oddsmall}.
\end{theorem}

When the set $S$ satisfies some additional conditions, the upper bounds on $u_0(S)$ in Theorem~\ref{thm-oddupper} can be further improved.

\begin{proposition}\label{prop-improvedupper}
Let $q$ be an odd prime power. Let $S$ be a $(q+1)$-set which is not a $(q+1)$-arc. Then we have the following.
\begin{itemize}
\item[(1)] If $S$ contains an internal nucleus and has a nucleus, then $u_0(S) \le \frac{q(q-1)}{2}-\frac{q-1}{2}$, where the equality holds if and only if $S$ is of the form in Example~\ref{exam-oddlarge}(2).
\item[(2)] If $S$ contains no internal nucleus and has a nucleus, then $u_0(S) < \frac{q(q-1)}{2}-\frac{q-1}{2}$.
\end{itemize}
\end{proposition}

\begin{remark}
When $S$ has no nucleus, as indicated by Example~\ref{exam-oddlarge}(1), the upper bound $u_0(S) \le \frac{q(q-1)}{2}-\frac{q-1}{2}$ in Proposition~\ref{prop-improvedupper} does not hold in general. Therefore, when $q$ is odd, we can only derive partial information about the second largest value of $u_0(S)$. It remains unclear whether the $(q+1)$-sets in Example~\ref{exam-oddlarge} always leads to the configurations achieving the second and third largest non-hitting index. As shown in Lemma~\ref{lem-oddsmall}, this is indeed the case when $q \in \{3,5,7\}$. In contrast, when $q$ is even, Theorem~\ref{thm-evenupper} indicates that the $(q+1)$-sets $S$ in Example~\ref{exam-evenlarge} always give the second and third largest value of $u_0(S)$, and very few $(q+1)$-sets $S$ which do not come from Example~\ref{exam-evenlarge} can achieve the same.
\end{remark}

Now, we proceed to derive a lower bound on $u_0(S)$ and give a characterization of $S$ when $u_0(S)$ is close to the lower bound.

\begin{theorem}\label{thm-lower}
Let $S$ be a $(q+1)$-set of degree $n$ with $3 \le n \le q+1$. Then we have $u_0(S) \ge n(q+2-n)-(q+1)$. More precisely,
\begin{itemize}
\item[(1)] $u_0(S)=0$ if and only if $n=q+1$, namely, $S$ forms a line in $PG(2,q)$. In this case, we have
$$
u_1(S)=q^2+q, \quad u_{q+1}(S)=1.
$$
\item[(2)] $u_0(S)=q-1$ if and only if $n=q$, namely, $q$ points of $S$ are on a line $\ell$ and the remaining one point is not on $\ell$. In this case, we have
$$
u_0(S)=q-1, \quad u_1(S)=q^2-q+1, \quad u_2(S)=q, \quad u_q(S)=1.
$$

\item[(3)] If $3 \le n \le q-1$, then $u_0(S) \ge 2q-4$. In particular,
   \begin{itemize}
      \item[(3a)] $u_0(S)=2q-4$ if and only if $n=q-1$, namely, $q-1$ points of $S$ are on one line, and the line determined by the remaining two points is a $3$-secant line to $S$. In this case, we have
$$
u_0(S)=2q-4, \quad u_1(S)=q^2-3q+7, \quad u_2(S)=2q-4, \quad u_3(S)=1, \quad u_{q-1}(S)=1.
$$
      \item[(3b)] $u_0(S)=2q-3$ if and only if $n=q-1$, namely, $q-1$ points of $S$ are on one line, and the line determined by the remaining two points is a $2$-secant line to $S$. In this case, we have
$$
u_0(S)=2q-3, \quad u_1(S)=q^2-3q+4, \quad u_2(S)=2q-1, \quad u_{q-1}(S)=1.
$$
   \end{itemize}
\end{itemize}
\end{theorem}

\begin{remark}\label{rem-nonhittingspec}
For a prime power $q$, define its non-hitting spectrum as
$$
\Spec(q)=\{ u_0(S) \mid \mbox{$S$ is a $(q+1)$-set in $PG(2,q)$}\}.
$$
Combining Proposition~\ref{prop-upperbound}, Examples~\ref{exam-evenlarge}, \ref{exam-oddlarge} and Theorems~\ref{thm-evenupper}, \ref{thm-oddupper}, \ref{thm-lower}, we have
$$
\Spec(2)=\{0,1\}, \quad \Spec(3)=\{0,2,3\}, \quad \Spec(4)=\{0,3,4,5,6\}, \quad \Spec(5)=\{0,4,6,7,8,9,10\}.
$$
In addition, when $q=7$, we have
$$
\Spec(7) \cap [0,12]=\{0,6,10,11,12\}, \quad \Spec(7) \cap [18,21]=\{18,19,21\},
$$
and the set $\Spec(7) \cap [13,17]$ is still open.
\end{remark}


\section{The polynomial aspect of the intersection distribution}\label{sec-3}

In Section~\ref{sec-2}, we used a geometric approach to characterize $(q+1)$-set $S$ by its non-hitting index $u_0(S)$, when $u_0(S)$ is very close to the lower and upper bounds. Except these extremal cases, we know very little about the non-hitting index and intersection distribution of $S$.
On the other hand, when $u_0(S)$ is far away from $0$ and $\frac{q(q-1)}{2}$, the geometric approach becomes increasingly complicated. This motivates us to consider the polynomial aspect of the intersection distribution, where an algebraic approach involving polynomials over finite fields comes into play. The polynomial viewpoint enables us to derive bounds on the non-hitting index of certain $(q+1)$-sets, and to determine the intersection distribution of several classes of $(q+1)$-sets, which have nice polynomial representations.

Recall that a polynomial $f$ over $\Fq$ corresponds to a $(q+1)$-set
\begin{equation}\label{eqn-q+1}
S_f=\{ \lan (x,f(x),1) \ran \mid x \in \Fq \} \cup \{ \lan (0,1,0) \ran \},
\end{equation}
in which $\lan (0,1,0) \ran$ is an internal nucleus of $S_f$. Conversely, for a $(q+1)$-set $S$ with an internal nucleus, the following lemma indicates that $S$ can be associated with a polynomial over the finite field $\Fq$ of degree at most $q-1$, by using a proper coordinate system. We note that $S$ has a polynomial representation \eqref{eqn-q+1} only if it contains an internal nucleus.
This assumption imposes no restriction to study Kakeya sets, which will be the main theme of Section~\ref{sec-4}. We use $\lan (a,b,c)^T \ran$ to denote a line in $PG(2,q)$ consisting of points $\lan (x,y,z) \ran$ satisfying $ax+by+cz=0$.

\begin{lemma}\label{lem-polyrep}
Let $S$ be a $(q+1)$-set in $PG(2,q)$ containing an internal nucleus. Then $PG(2,q)$ can be coordinatized, such that $S=S_f$, for some polynomail $f$ over $\Fq$ of degree at most $q-1$ and $\lan (0,1,0) \ran$ is an internal nucleus of $S_f$. Moreover, $S_f$ has a nucleus $\lan (1,c,0) \ran$, where $c \in \Fq$, if and only if $f(x)-cx$ is a permutation polynomial.
\end{lemma}
\begin{proof}
Suppose $O$ is an internal nucleus of $S$, then there are $q$ lines through $O$ intersecting $S$ in a second point and one line $\ell$ which does not. Assume that $\ell$ is the line at infinity $\lan (0,0,1)^T \ran$ and $O=\lan (0,1,0) \ran$. Except $\ell$, the remaining $q$ lines through $O$ are of the form $\lan (1,0,-x)^T \ran$, where $x \in \Fq$. Every point of $S$ on $\lan (1,0,-x)^T \ran$ other than $O$ has the form $\lan (x,y_x,1) \ran$, where $y_x$ is an element of $\Fq$ depending on $x$. By Lagrange interpolation \cite[Theorem 1.71]{LN}, there exists a polynomial $f$ over $\Fq$ of degree at most $q-1$, such that $f(x)=y_x$ for each $x \in \Fq$. Consequently, $S=\{\lan (x,f(x),1) \ran \mid x \in \Fq\} \cup \{ \lan (0,1,0) \ran \}$.

Clearly, a nucleus to $S$ must be on the line $\ell$ and of the form $\lan (1,c,0) \ran$ for some $c \in \Fq$. Note that the line through $\lan (1,c,0) \ran$ and $\lan (x,f(x),1) \ran$ is $\lan (c,-1,f(x)-cx)^T \ran$. Hence, $\lan (1,c,0) \ran$ is a nucleus to $S$ if and only if $f(x)-cx$ is a permutation polynomial.
\end{proof}


The following is a natural connection between the intersection distributions of $f$ and the corresponding $(q+1)$-set $S_f$.

\begin{proposition}\label{prop-intdispoly}
Let $f$ be a polynomial over $\Fq$ and $S_f$ the associated $(q+1)$-set in $PG(2,q)$. Then we have
\begin{align*}
& v_0(f)=u_0(S_f), \quad v_1(f)=u_1(S_f)-1, \quad v_2(f)=u_2(S_f)-q, \\
& v_i(f)=u_i(S_f),\; \mbox{for each $3 \le i \le q$}, \quad u_{q+1}(S_f)=0.
\end{align*}
Consequently, we have
\begin{itemize}
\item[(1)] $v_0(f) \ge q-1$, where $v_0(f)=q-1$ if and only if $f$ is a linear function.
\item[(2)] $v_0(f) \le \frac{q(q-1)}{2}$, where $v_0(f)=\frac{q(q-1)}{2}$ if and only if one of the following holds:
           \begin{itemize}
             \item[$\bullet$] When $q$ is even, $f(x)-cx$ is an o-polynomial for some $c \in \Fq$.
             \item[$\bullet$] When $q$ is odd, $f$ is projectively equivalent to $x^2$.
           \end{itemize}
\item[(3)] If $q$ is even and $f(x)-cx$ is not an o-polynomial for each $c \in \Fq$, then $v_0(f) \le \frac{q(q-1)}{2}-\frac{q}{2}+1$.
\item[(4)] If $q$ is odd and $f(x)-cx$ is a permutation polynomial for some $c \in \Fq$, then $v_0(f) \le \frac{q(q-1)}{2}-\frac{q-1}{2}$.
\end{itemize}
\end{proposition}
\begin{proof}
Since $\lan (0,1,0) \ran$ is an internal nucleus of $S_f$, the relation between $v_i(f)$ and $u_i(S_f)$ easily follows. Note that $v_0(f)=u_0(S_f)$. The $(q+1)$-set $S_f$ containing an internal nucleus means that its degree is at most $q$. Therefore, (1) follows from Theorem~\ref{thm-lower}(2). By Proposition~\ref{prop-upperbound}, $u_0(S_f)=\frac{q(q-1)}{2}$ if and only if $S_f$ is a $(q+1)$-arc in $PG(2,q)$. If $q$ is even, then a $(q+1)$-arc $S_f$ can be uniquely extended to a $(q+2)$-arc $S_f \cup \{ \lan (1,c,0) \ran \}$, where $c \in \Fq$ \cite[Corollary 8.7]{Hirs}. According to \cite[Theorem 8.22]{Hirs}, \cite[Lemma 13]{CM} and Lemma~\ref{lem-polyrep}, $S_f \cup \{\lan (1,c,0) \ran\}$ is a $(q+2)$-arc if and only if $f(x)-cx$ is an o-polynomial. If $q$ is odd, there exists an automorphism of $PG(2,q)$, which transforms $S_f$ into $S_{x^2}=\{ \lan (x,x^2,1) \ran \mid x \in \Fq \} \cup \{ \lan (0,1,0) \ran \}$ \cite[Theorem 8.14]{Hirs}. Thus, we complete the proof of (2). When $q$ is even, if $f(x)-cx$ is not an o-polynomial for each $c \in \Fq$, then $S_f \cup \{\lan (1,c,0) \ran\}$ is not a $(q+2)$-arc for each $c \in \Fq$, which implies that $S_f$ is not a $(q+1)$-arc. Hence, (3) follows from Theorem~\ref{thm-evenupper}. Finally, if $f(x)-cx$ is a permutation polynomial, then by Lemma~\ref{lem-polyrep}, $\lan (1,c,0) \ran$ is a nucleus to $S_f$, and therefore, $v_0(f)=u_0(S_f)=u_0(S_f \cup \{\lan (1,c,0) \ran\})$. When $q$ is odd, applying Proposition~\ref{prop-improvedupper} finishes the proof of (4).
\end{proof}

\begin{remark}
Historically, various aspects of power mappings, including planarity {\rm\cite{CM2,Pott}}, almost perfect nonlinearity {\rm\cite{BCCL,HRS,Pott}},  differential properties in general {\rm\cite{BCC}}, bent property {\rm\cite{CCK,LL}} and Walsh spectrum {\rm\cite{Hel}}, have been intensively studied. We regard the non-hitting index as a new way to analyze and distinguish power mappings. More precisely, given a polynomial $f$ over $\Fq$, its non-hitting index $v_0(f)$ belongs to the interval $[q-1,\frac{q(q-1)}{2}]$, where $v_0(f)=q-1$ if and only if $f$ is linear, and $v_0(f)=\frac{q(q-1)}{2}$ if and only if $f$ is an o-polynomial when $q$ is even, or $f$ is projectively equivalent to $x^2$ when $q$ is odd. The non-hitting index of a polynomial $f$ proposes a new viewpoint to measure the distance between $f$ and the aforementioned polynomials.
\end{remark}

Proposition~\ref{prop-intdispoly} indicates a polynomial approach to study the intersection distribution of a $(q+1)$-set $S$ with an internal nucleus. Once we know the polynomial $f$ associated with $S$, determining the intersection distribution of $S$ can be converted to the computation of the intersection distribution of $f$, which is a problem concerning polynomials over finite field. Not surprisingly, using known results related to polynomials, we can obtain more detailed information about lower and upper bounds on $v_0(f)$, which involves the degree of $f$ and the size of $N_f$.

\begin{proposition}\label{prop-polybound}
Let $f$ be a polynomial over $\Fq$ of degree $d$ with $2 \le d \le q-1$. Then we have
\begin{equation}\label{eqn-lowerupper1}
\lc \frac{q-1}{d}\rc (q-|N_f|) \le v_0(f) \le (q-\lc \frac{q}{d} \rc)(q-|N_f|).
\end{equation}
In particular, we have the following bounds which only involve the degree $d$ and the finite field size $q$.
\begin{itemize}
\item[(1)]
\begin{equation}\label{eqn-lowerupper2}
v_0(f) \ge \lc \frac{q-1}{d}\rc \max \{ \lc \frac{q-1}{d-1} \rc, d+1 \}
\end{equation}
\item[(2)] If $d | (q-1)$ then
\begin{equation}\label{eqn-lowerupper3}
\frac{q(q-1)}{d} \le v_0(f) \le \frac{(d-1)q(q-1)}{d}.
\end{equation}
\end{itemize}
\end{proposition}
\begin{proof}
By definition,
\begin{equation}\label{eqn-nonhitting}
v_0(f)=\sum_{c \in \Fq} (q-|V_{f,c}|)=q^2-q|N_f|-\sum_{c \notin N_f} |V_{f,c}|.
\end{equation}
Note that for each polynomial $g$ over $\Fq$ of degree $d$, which is not a permutation polynomial, we have $\lc \frac{q}{d}\rc \le |V_g| \le \lf q-\frac{q-1}{d} \rf$ (see for instance \cite[p. 711]{WSC}). Hence, by \eqref{eqn-nonhitting}, we have
\begin{align*}
v_0(f) &\ge q^2-q|N_f|-(q-|N_f|)\lf q-\frac{q-1}{d} \rf=\lc \frac{q-1}{d}\rc (q-|N_f|), \\
v_0(f) &\le q^2-q|N_f|-(q-|N_f|)\lc \frac{q}{d} \rc=(q-\lc \frac{q}{d} \rc)(q-|N_f|).
\end{align*}
Note that $|N_f| \le \min\{ q-1-d, q-\lc \frac{q-1}{d-1} \rc\}$ (see for instance \cite[pp. 133-134]{WMS}), hence we have $v_0(f) \ge \lc \frac{q-1}{d}\rc \max \{d+1, \lc \frac{q-1}{d-1} \rc \}$. Finally, if $d|(q-1)$ and $d>1$, then there exists no permutation polynomial of degree $d$ \cite[Corollary 7.5]{LN}. Hence, $|N_f|=0$ and we derive \eqref{eqn-lowerupper3} from \eqref{eqn-lowerupper1}.
\end{proof}

\begin{remark}
Except for a few bounds {\rm\cite{EGN,Tur,WMS,WSC}}, we do not know much about the size of $N_f$. This is the reason that the tightness of the lower and upper bounds in \eqref{eqn-lowerupper1} is subtle. Still, the two special cases in \eqref{eqn-lowerupper2} and \eqref{eqn-lowerupper3} give us some clue. On one hand, when $q$ is odd and $d=2$, the lower and upper bounds in \eqref{eqn-lowerupper3} coincide and are both tight. On the other hand, the lower bound in \eqref{eqn-lowerupper2} is a constant multiple of $q$ when $d$ approaches $q-1$, which is weak in general. Moreover, the upper bound in \eqref{eqn-lowerupper3} becomes vacuous when $d>2$. In this sense, we think the bounds in \eqref{eqn-lowerupper1} could be further improved.
\end{remark}


Given an arbitrary $(q+1)$-set $S$ with an internal nucleus, it is in general difficult to compute its intersection distribution. Equivalently, computing the intersection distribution of an arbitrary polynomial $f$ is hard. On the other hand, we recall that by Remark~\ref{rem-polyint}(2), the intersection distribution of $f$ follows immediately from the multiplicity distribution of $f$ at $c$, where $c$ ranges over $\Fq$. We list in Table~\ref{tab-interdispoly} the intersection distributions of several classes of power mappings, which are considered to be the most obvious ones. More precisely, in Appendix~\ref{sec-appendixB}, we compute the multiplicity distribution of power mappings $x^d$ over finite fields $\Fq$ with $q=p^s$ and $p$ being prime, where $d \in \{p^i,p^i+1,\frac{q-1}{2},\frac{q+1}{2},q-2,q-1\}$, $0 \le i \le s-1$. Thus, the intersection distribution in Table~\ref{tab-interdispoly} follows immediately. Consequently, we derive the intersection distribution of the corresponding $(q+1)$-set which has a nice polynomial representation. We note that to our best knowledge, there are very few polynomials whose multiplicity distribution has been known before, see for instance \cite{Blu} and \cite[Chapter 3, Section 4]{LN}.

\begin{table}[H]
\ra{1.5}
\caption{The intersection distribution of some power mappings $x^d$ over $\Fq$, where $q=p^s$ for a prime $p$}
\label{tab-interdispoly}
\begin{center}

\begin{tabular}{|c|c|}
\hline
Exponent &Intersection Distribution \\ \hline
$d=p^i$, $0 \le i \le s-1$ & \multirow{2}{*}{$v_0=p^{s-h}(p^s-1)$, $v_1=\frac{p^s(p^{s+h}-2p^s+1)}{p^h-1}$, $v_{p^h}=\frac{p^{s-h}(p^s-1)}{p^h-1}$} \\
$h=\gcd(i,s)$      & \\ \hline
$d=p^i+1$, $0 \le i \le s-1$ & $v_0=\frac{p^{h}(p^{2s}-1)}{2(p^h+1)}$, $v_1=p^{2s-h}-p^{s-h}+1$, \\
$h=\gcd(i,s)$  & $v_2=\frac{p^h(p^s-2p^{s-h}+1)(p^s-1)}{2(p^h-1)}$, $v_{p^h+1}=\frac{(p^{s-h}-1)(p^s-1)}{p^{2h}-1}$ \\ \hline
$d=\frac{q-1}{2}$, $q \equiv 1 \bmod4$ & $v_0=\frac{q^2+6q-15}{4}$, $v_1=\frac{q^2-4q+5}{2}$, $v_2=\frac{q^2+2q-3}{4}$, $v_{\frac{q-1}{2}}=2$ \\ \hline
$d=\frac{q-1}{2}$, $q \equiv 3 \bmod4$ & $v_0=\frac{q^2+4q-13}{4}$, $v_1=\frac{q^2-q+2}{2}$, $v_2=\frac{q^2-4q+3}{4}$, $v_3=\frac{q-1}{2}$, $v_{\frac{q-1}{2}}=2$ \\ \hline
$d=\frac{q+1}{2}$& $v_0=\frac{q^2+2q-3}{4}$, $v_1=\frac{q^2-3}{2}$, $v_2=\frac{(q-1)^2}{4}$, $v_{\frac{q+1}{2}}=2$ \\ \hline
$d=q-2$, $q$ even & $v_0=\frac{q(q-1)}{2}$, $v_1=q$, $v_2=\frac{q(q-1)}{2}$ \\ \hline
$d=q-2$, $q$ odd & $v_0=\frac{(q-1)^2}{2}$, $v_1=\frac{5q-3}{2}$, $v_2=\frac{(q-1)(q-3)}{2}$, $v_3=\frac{q-1}{2}$ \\ \hline
$d=q-1$ & $v_0=2q-3$, $v_1=q^2-3q+3$, $v_2=q-1$, $v_{q-1}=1$ \\ \hline
\end{tabular}
\end{center}
\end{table}

\begin{remark}\label{rem-measureconnection}
The equivalence problem for polynomials has been intensively studied under various concepts of equivalence, such as the extended-affine (EA) equivalence (see for instance {\rm\cite[p. 1142]{BCP}}) and the Carlet-Charpin-Zinoviev (CCZ) equivalence {\rm\cite[Proposition 3]{CCZ}, \cite[Definition 1]{BCP}}. These two equivalence criteria are based on the fact that by choosing a basis of $\F_{p^s}$ over $\Fp$, a polynomial over $\F_{p^s}$ can be represented as a mapping from vector space $\Fp^s$ to $\Fp^s$. In the definition of EA and CCZ equivalence, the structure of the vector space $\Fp^s$ plays a crucial role. In contrast, the projective equivalence provides a new angle for the equivalence problem, in the sense that it is really a property about the finite field $\F_{p^n}$ and has nothing to do with the vector space $\Fp^n$. More precisely, the projective equivalence only depends on a geometrical property of the corresponding $(p^s+1)$-set in $PG(2,p^s)$. Note that the intersection distribution is an invariant of the projective equivalence. In Table~\ref{tab-interdispoly}, the intersection distribution offers an interesting viewpoint to distinguish projectively inequivalent power mappings of the form $x^{p^i}$ and $x^{p^i+1}$ over $\F_{p^s}$.
\end{remark}

Finally, in Table~\ref{tab-monononhitting}, we document some computational results demonstrating the non-hitting index of all power mappings in $\Fq$ of degree at most $q-1$, where $q \le 16$. For $1 \le d \le q-1$ and $\gcd(d,q-1)=1$, by definition, the two power mappings $x^d$ and $x^{d^{-1}}$ have the same multiplicity distribution, where $d^{-1}$ is the inverse of $d$ modulo $q-1$. Consequently, $x^d$ and $x^{d^{-1}}$ have the same intersection distribution. Hence, in the second column, if $\gcd(d,q-1)=1$, we group the two exponents $\{d,d^{-1}\}$ since they have the same non-hitting index. We use the superscript $\bigstar$ to highlight the non-hitting index which does not follow from Table~\ref{tab-interdispoly}. 

\begin{table}
\ra{1.2}
\begin{center}
\caption{The non-hitting index of all power mappings in $\Fq$, $q \le 16$}
\label{tab-monononhitting}
\begin{tabular}{|c|l|}
\hline
$q$ & $(d,v_0(x^d))$ \\ \hline
$2$  &  $(1,1)$   \\ \hline
$3$  &  $(1,2)$, $(2,3)$   \\ \hline
$4$  &  $(1,3)$, $(2,6)$, $(3,5)$   \\ \hline
$5$  &  $(1,4)$, $(2,10)$, $(3,8)$, $(4,7)$   \\ \hline
$7$  &  $(1,6)$, $(2,21)$, $(3,16)$, $(4,15)$, $(5,18)$, $(6,11)$ \\ \hline
$8$  &  $(1,7)$, $(\{2,4\},28)$, $(\{3,5\},21)$, $(6,28)$, $(7,13)$ \\ \hline
$9$  &  $(1,8)$, $(2,36)$, $(3,24)$, $(4,30)$, $(5,24)$, $(6,28)^\bigstar$, $(7,32)$, $(8,15)$  \\ \hline
$11$  &  $(1,10)$, $(2,55)$, $(\{3,7\},40)^\bigstar$, $(4,45)^\bigstar$, $(5,38)$, $(6,35)$ , $(8,45)^\bigstar$, $(9,50)$, $(10,19)$ \\ \hline
\multirow{2}{*}{$13$}  &  $(1,12)$, $(2,78)$, $(3,56)^\bigstar$, $(4,57)^\bigstar$, $(5,60)^\bigstar$, $(6,58)$, $(7,48)$, $(8,69)^\bigstar$, $(9,56)^\bigstar$, \\
& $(10,54)^\bigstar$, $(11,72)$, $(12,23)$ \\ \hline
\multirow{2}{*}{$16$}  &  $(1,15)$, $(\{2,8\},120)$, $(3,85)$, $(4,60)$, $(5,102)$, $(6,85)^\bigstar$, $(\{7,13\},75)^\bigstar$, $(9,85)$, \\
& $(10,87)^\bigstar$, $(11,90)^\bigstar$, $(12,70)^\bigstar$, $(14,120)$, $(15,29)$ \\ \hline
\end{tabular}
\end{center}
\end{table}

\section{Application to Kakeya sets in affine planes}\label{sec-4}

Let $\ell$ be the line at infinity in $PG(2,q)$. For each point $P \in \ell$, define $\ell_P$ to be a line through $P$ other than $\ell$. A \emph{Kakeya set} in $PG(2,q)$ is defined to be the point set
$$
K=(\bigcup_{P \in \ell} \ell_P) \sm \ell.
$$
If we restrict ourselves to the affine plane $AG(2,q)=PG(2,q) \sm \ell$, then the Kakeya set $K$ contains an affine line in each direction. While the construction of Kakeya sets is easy, computing its size is difficult. The Kakeya set problem asks for the construction and characterization of Kakeya sets in $PG(2,q)$ with small or large sizes, for which there have been a series of results \cite{BDMS,BM,BV,DM1,DM2,DMS,Fab}. More precisely, when $q$ is even, the Kakeya sets with first, second, and third smallest sizes have been characterized \cite{BDMS,BM}. When $q$ is odd, the characterization of Kakeya sets with smallest size is also known \cite{BM}. On the other hand, in \cite{DM2}, Dover and Mellinger did an exhaustive computer search and determined all attainable sizes of Kakeya sets in $PG(2,q)$ with $q \le 9$. Except for the Kakeya sets with very small or large sizes close to the lower or upper bounds, a large fraction of them have not been well understood. This actually motivates us to consider the Kakeya set problem in a different way. Inspired by the interaction between $(q+1)$-sets and polynomials in Proposition~\ref{prop-intdispoly}, we propose a polynomial approach to Kakeya sets. By examining Kakeya sets which have nice algebraic structure, that is, derived from monomials, we succeed in computing their sizes in many cases. In fact, the computation essentially boils down to the calculation of the multiplicity distribution of certain monomials. As a consequence, many attainable sizes in \cite[Table 1]{DM2} can be explained by Kakeya sets derived from monomials.

We shall note that instead of affine planes, Kakeya sets in affine spaces with higher dimension have also been studied, see \cite{Dvir,Wol} for instance. Here, we focus on the affine plane case, for which the following alternative viewpoint that appeared in \cite{BM} is crucial.

\begin{lemma}\label{lem-DK}
Let $K$ be a Kakeya set in $PG(2,q)$, where $K=(\bigcup_{P \in \ell} \ell_P) \sm \ell$. Define the dual Kakeya set $DK$ to be the dual of the $q+2$ lines $\{\ell_P \mid P \in \ell\} \cup \{\ell\}$. Then $DK$ is a $(q+2)$-set in $PG(2,q)$ with an internal nucleus, such that $|K|=q^2-u_0(DK)$.
\end{lemma}
\begin{proof}
Clearly, $DK$ is a $(q+2)$-set and the point dual to $\ell$ is an internal nucleus. Note that $|K|=|K \cup \ell |-(q+1)$. By duality, $|K \cup \ell|$ equals the number of lines intersecting the dual $DK$, which is $q^2+q+1-u_0(DK)$. Thus, $|K|=|K \cup \ell |-(q+1)=q^2-u_0(DK)$.
\end{proof}

\begin{remark}
The equation $|K|=q^2-u_0(DK)$ indicates that computing the size of a Kakeya set $K$ is equivalent to computing the non-hitting index of the dual Kakeya set $DK$. There has been some recent progress on characterizing Kakeya sets by their sizes {\rm\cite{BDMS,BM,DM2}}. From the perspective of Lemma~\ref{lem-DK}, these work succeeded in characterizing $(q+2)$-sets with an internal nucleus by its non-hitting index.
\end{remark}

Given a polynomial $f$ over $\Fq$, let $S_f$ be the $(q+1)$-set defined in \eqref{eqn-q+1}. For $c \in \Fq$, by adding a point $\lan (1,c,0) \ran$ to $S_f$, we obtain a $(q+2)$-set with an internal nucleus $\lan (0,1,0) \ran$, which is a dual Kakeya set
$$
DK(f,c):=S_f \cup \{\lan (1,c,0) \ran\}=\{ \lan (x,f(x),1) \ran \mid x \in \Fq \} \cup \{ \lan (0,1,0),(1,c,0)\ran \}.
$$
The next proposition, together with Remark~\ref{rem-polyint}(2), shows that the intersection distribution of $DK(f,c)$ follows from the multiplicity distribution of $f$.

\begin{proposition}\label{prop-intdisKconnection}
For a polynomial $f$ over $\Fq$ and $c \in \Fq$, the intersection distribution of $DK(f,c)$ is as follows.
\begin{align*}
&u_0(DK(f,c))=v_0(f)-M_0(f,c), \\
&u_1(DK(f,c))=v_1(f)-M_1(f,c)+M_0(f,c), \\
&u_2(DK(f,c))=v_2(f)-M_2(f,c)+M_1(f,c)+q+1, \\
&u_i(DK(f,c))=v_i(f)-M_i(f,c)+M_{i-1}(f,c), \,\mbox{for each $3 \le i \le q$}, \\
&u_{q+1}(DK(f,c))=M_q(f,c).
\end{align*}
\end{proposition}
\begin{proof}
Embed the graph of $f$ into $PG(2,q)$, so that $G_f=\{\lan (x,f(x),1) \ran \mid x \in \Fq\}$. To compute the intersection distribution of $DK(f,c)$ from that of $f$, it suffices to consider how the $2q+1$ lines through $\lan (0,1,0)\ran$ and $\lan (1,c,0) \ran$ intersect $G_f$. The two points $\lan (0,1,0)\ran$ and $\lan (1,c,0) \ran$ determine the line at infinity $\ell$. Except $\ell$, there are $q$ lines $\{\lan (1,0,b)^T \ran \mid b \in \Fq\}$ trough $\lan (0,1,0)\ran$, so that each of them intersects $G_f$ in one point. This explains the $q+1$ term in the expression of $u_2(DK(f,c))$. Except $\ell$, there are $q$ lines $\{\lan (-c,1,-b)^T \ran \mid b \in \Fq\}$ through $\lan (1,c,0) \ran$. By definition, for each $0 \le i \le q$, there are exactly $M_i(f,c)$ lines in the set $\{\lan (-c,1,-b)^T \ran \mid b \in \Fq\}$, intersecting $G_f$ in $i$ points. Meanwhile, each of these $M_i(f,c)$ lines intersects $DK(f,c)$ in $i+1$ points. This explains the $M_i(f,c)$ terms in the intersection distribution.
\end{proof}

\ra{1.5}
\begin{longtable}{|c|c|c|c|}
\caption{The intersection distribution of the dual Kakeya set $DK(d,c)$ and the size of the Kakeya set $K$, where $c \in \Fq$ and $q=p^s$ for a prime $p$}
\label{tab-interdisKakeya}
\\ \hline
$d$ & $c$ & Intersection Distribution of $DK(d,c)$ and $|K|$ \\ \hline
$d=p^i$ & \multirow{2}{*}{$c \in C_{0}^{(p^h-1,q)}$}  & $u_0=p^{s}(p^{s-h}-1)$, $u_1=\frac{p^{s-h}(p^{s+2h}-2p^{s+h}+p^{2h}-p^h+1)}{p^h-1}$, $u_2=p^s+1$, \\
$0 \le i \le s-1$ & & $u_{p^h}=\frac{p^s(p^{s-h}-1)}{p^h-1}$, $u_{p^h+1}=p^{s-h}$, $|K|=p^{2s}-p^{2s-h}+p^s$ \\ \cline{2-3}
$h=\gcd(i,s)$ & \multirow{2}{*}{$c \notin C_{0}^{(p^h-1,q)}$}  &  $u_0=p^{s-h}(p^s-1)$, $u_1=\frac{p^s(p^s-1)(p^h-2)}{p^h-1}$, $u_2=2p^s+1$, \\
& & $u_{p^h}=\frac{p^{s-h}(p^s-1)}{p^h-1}$, $|K|=p^{2s}-p^{2s-h}+p^{s-h}$ \\ \hline
&  \multirow{3}{*}{$c=0$}  & $u_0=\frac{p^h(p^s-1)^2}{2(p^h+1)}$, $u_1=\frac{(p^s+p^{s-h}+p^h)(p^s-1)}{p^h+1}$, \\
$d=p^i+1$  &  & $u_2=\frac{p^{2s+h}-2p^{2s}+2p^{s+h}+3p^h-4}{2(p^h-1)}$, $u_{p^h+1}=\frac{(p^{s-h}-p^h)(p^s-1)}{p^{2h}-1}$, \\
$0 \le i \le s-1$ & & $u_{p^h+2}=\frac{p^s-1}{p^h+1}$, $|K|=p^{2s}-\frac{p^h(p^s-1)^2}{2(p^h+1)}$ \\ \cline{2-3}
$l_2(i)<l_2(s)$ &  \multirow{4}{*}{$c\ne0$}  & $u_0=\frac{p^{s+h}(p^s-1)}{2(p^h+1)}$, $u_1=\frac{p^h(p^s-1)}{2(p^h+1)}+p^{2s-h}-2p^{s-h}+1$, \\
$h=\gcd(i,s)$ &   & $u_2=\frac{(p^h-2)(p^s-1)(p^s-2)}{2(p^h-1)}+2p^s+p^{s-h}-1$, $u_3=\frac{p^{s+h}-2p^s+p^h}{2(p^h-1)}$, \\
&   & $u_{p^h+1}=\frac{p^{2s-h}-p^s-2p^{s-h}+p^h+1}{p^{2h}-1}$, $u_{p^h+2}=\frac{p^{s-h}-p^h}{p^{2h}-1}$, \\
&   & $|K|=p^{2s}-\frac{p^{s+h}(p^s-1)}{2(p^h+1)}$ \\ \hline
$d=p^i+1$&  \multirow{2}{*}{$c=0$}  & $u_0=\frac{p^{h}(p^{2s}-1)}{2(p^h+1)}$, $u_1=(p^{s-h}-1)(p^s-1)$, $u_2=\frac{(p^h-2)(p^s-1)^2}{2(p^h-1)}+3p^s$, \\
$p=2$ &   & $u_{p^h+1}=\frac{(p^{s-h}-1)(p^s-1)}{p^{2h}-1}$, $|K|=p^{2s}-\frac{p^{h}(p^{2s}-1)}{2(p^h+1)}$ \\ \cline{2-3}
$0 \le i \le s-1$ &  \multirow{3}{*}{$c\ne0$}  & $u_0=\frac{p^{h}(p^s+1)(p^s-2)}{2(p^h+1)}$, $u_1=\frac{p^h(p^s+1)}{2(p^h+1)}+p^{2s-h}-2p^{s-h}+2$, \\
$l_2(i) \ge l_2(s)$&   & $u_2=\frac{(p^h-2)(p^s-1)(p^s-2)}{2(p^h-1)}+2p^s+p^{s-h}-2$, $u_3=\frac{p^{s+h}-2p^s+p^h}{2(p^h-1)}$, \\
$h=\gcd(i,s)$ &   & $u_{p^h+1}=\frac{(p^{s-h}-1)(p^s-2)}{p^{2h}-1}$, $u_{p^h+2}=\frac{p^{s-h}-1}{p^{2h}-1}$, $|K|=p^{2s}-\frac{p^{h}(p^s+1)(p^s-2)}{2(p^h+1)}$ \\ \hline
$d=p^i+1$&  \multirow{3}{*}{$c=0$}  & $u_0=\frac{(p^{s+h}-1)(p^s-1)}{2(p^h+1)}$, $u_1=\frac{(2p^{s-h}+1)(p^s-1)}{2}$, \\
$p$ odd&   & $u_2=\frac{p^{2s+h}-2p^{2s}+p^{s+h}+p^s+4p^h-5}{2(p^h-1)}$, $u_3=\frac{p^s-1}{2}$, \\
$0 \le i \le s-1$& & $u_{p^h+1}=\frac{(p^{s-h}-1)(p^s-1)}{p^{2h}-1}$, $|K|=p^{2s}-\frac{(p^{s+h}-1)(p^s-1)}{2(p^h+1)}$ \\ \cline{2-3}
$l_2(i) \ge l_2(s)$ &  \multirow{3}{*}{$c\ne0$}  & $u_0=\frac{p^{2s+h}-p^{s+h}-p^h+1}{2(p^h+1)}$, $u_1=\frac{p^{s+h}-1}{2(p^h+1)}+p^{2s-h}-2p^{s-h}+1$, \\
$h=\gcd(i,s)$&   & $u_2=\frac{p^{2s+h}-2p^{2s}+p^{s+h}+4p^s-2p^{s-h}+p^h-3}{2(p^h-1)}$, $u_3=\frac{p^{s+h}-2p^s+1}{2(p^h-1)}$, \\
&   & $u_{p^h+1}=\frac{(p^{s-h}-1)(p^s-2)}{p^{2h}-1}$, $u_{p^h+2}=\frac{p^{s-h}-1}{p^{2h}-1}$, $|K|=p^{2s}-\frac{p^{2s+h}-p^{s+h}-p^h+1}{2(p^h+1)}$ \\ \hline
$d=\frac{q-1}{2}$ & $c=0$ & $u_0=\frac{q^2+2q-3}{4}$, $u_1=\frac{q^2-2q-3}{2}$, $u_2=\frac{q^2+6q+5}{4}$, $u_{\frac{q+1}{2}}=2$, $|K|=\frac{3q^2-2q+3}{4}$ \\ \cline{2-3}
$q \equiv 1 \bmod4$  & \multirow{2}{*}{$c \ne 0$} & $u_0=\frac{q^2+5q-18}{4}$, $u_1=\frac{2q^2-9q+19}{4}$, $u_2=\frac{q^2+7q-8}{4}$, \\
& & $u_3=\frac{q+3}{4}$, $u_{\frac{q-1}{2}}=2$, $|K|=\frac{3q^2-5q+18}{4}$  \\ \hline
$d=\frac{q-1}{2}$ & \multirow{2}{*}{$c=0$} & \multirow{2}{*}{$u_0=\frac{q^2-1}{4}$, $u_1=\frac{q^2+q-6}{2}$, $u_2=\frac{q^2+11}{4}$, $u_3=\frac{q-1}{2}$, $u_{\frac{q+1}{2}}=2$, $|K|=\frac{3q^2+1}{4}$}  \\
$q \equiv 3 \bmod4$  &  &  \\ \hline
  & \multirow{2}{*}{$c \in C_0^{(2,q)}$} & $u_0=\frac{q^2+3q-18}{4}$, $u_1=\frac{2q^2-3q+15}{4}$, $u_2=\frac{q^2+q+4}{4}$,\\
$d=\frac{q-1}{2}$ & & $u_3=\frac{3q-9}{4}$, $u_4=1$, $u_{\frac{q-1}{2}}=2$, $|K|=\frac{3q^2-3q+18}{4}$  \\ \cline{2-3}
$q \equiv 3 \bmod8$ & \multirow{2}{*}{$c \in C_1^{(2,q)}$} & $u_0=\frac{q^2+3q-10}{4}$, $u_1=\frac{2q^2-3q-5}{4}$, $u_2=\frac{q^2+q+16}{4}$,\\
& & $u_3=\frac{3q-5}{4}$, $u_{\frac{q-1}{2}}=2$, $|K|=\frac{3q^2-3q+10}{4}$ \\ \hline

 & \multirow{2}{*}{$c \in C_0^{(2,q)}$} & $u_0=\frac{q^2+3q-14}{4}$, $u_1=\frac{2q^2-3q+3}{4}$, $u_2=\frac{q^2+q+16}{4}$, \\
$d=\frac{q-1}{2}$ & & $u_3=\frac{3q-13}{4}$, $u_4=1$, $u_{\frac{q-1}{2}}=2$, $|K|=\frac{3q^2-3q+14}{4}$ \\ \cline{2-3}
$q \equiv 7 \bmod8$ & \multirow{2}{*}{$c \in C_1^{(2,q)}$} & $u_0=\frac{q^2+3q-14}{4}$, $u_1=\frac{2q^2-3q+7}{4}$, $u_2=\frac{q^2+q+4}{4}$,\\ & & $u_3=\frac{3q-1}{4}$, $u_{\frac{q-1}{2}}=2$, $|K|=\frac{3q^2-3q+14}{4}$ \\ \hline

$d=\frac{q+1}{2}$ & $c \in \{0\} \cup$ & \multirow{4}{*}{$u_0=\frac{q^2+2q-3}{4}$, $u_1=\frac{q^2-2q-3}{2}$, $u_2=\frac{q^2+6q+5}{4}$, $u_{\frac{q+1}{2}}=2$, $|K|=\frac{3q^2-2q+3}{4}$}   \\
$q \equiv 1 \bmod4$ & $C_{0,0}^q \cup C_{1,1}^q$ & \\ \cline{1-2}
$d=\frac{q+1}{2}$ & $c \in C_{0,1}^q \cup$ & \\
$q \equiv 3 \bmod4$ & $C_{1,0}^q$ & \\ \hline

$d=\frac{q+1}{2}$ & $c \in C_{0,1}^q \cup$ & \multirow{4}{*}{$u_0=\frac{q^2-1}{4}$, $u_1=\frac{q^2+q-6}{2}$, $u_2=\frac{q^2+11}{4}$, $u_3=\frac{q-1}{2}$, $u_{\frac{q+1}{2}}=2$, $|K|=\frac{3q^2+1}{4}$} \\
$q \equiv 1 \bmod4$ & $C_{1,0}^q$ & \\ \cline{1-2}
$d=\frac{q+1}{2}$ &$c \in \{0\} \cup$ & \\
$q \equiv 3 \bmod4$ & $C_{0,0}^q \cup C_{1,1}^q$ & \\ \hline

$d=\frac{q+1}{2}$ & \multirow{2}{*}{$c=\pm1$} & \multirow{2}{*}{$u_0=\frac{q^2-1}{4}$, $u_1=\frac{q^2-3}{2}$, $u_2=\frac{q^2+4q+3}{4}$, $u_{\frac{q+1}{2}}=1$, $u_{\frac{q+3}{2}}=1$, $|K|=\frac{3q^2+1}{4}$} \\
$q$ odd & & \\ \hline

$d=q-2$ & $c=0$ & $u_0=\frac{q(q-1)}{2}$, $u_2=\frac{(q+1)(q+2)}{2}$, $|K|=\frac{q(q+1)}{2}$  \\ \cline{2-3}
$q$ even & $c \ne 0$ & $u_0=\frac{q(q-2)}{2}$, $u_1=\frac{3q}{2}$, $u_2=\frac{q^2}{2}+1$, $u_3=\frac{q}{2}$, $|K|=\frac{q(q+2)}{2}$ \\ \hline

 & $c=0$ & $u_0=\frac{(q-1)^2}{2}$, $u_1=\frac{3(q-1)}{2}$, $u_2=\frac{q^2+5}{2}$, $u_3=\frac{q-1}{2}$, $|K|=\frac{q^2+2q-1}{2}$ \\ \cline{2-3}
$d=q-2$ & \multirow{2}{*}{$c \in C_0^{(2,q)}$} & $u_0=\frac{(q-1)(q-2)}{2}$, $u_1=3q-4$, $u_2=\frac{q^2-3q+14}{2}$, \\
$q \equiv 1 \bmod{4}$ & & $u_3=q-4$, $u_4=1$, $|K|=\frac{q^2+3q-2}{2}$ \\ \cline{2-3}
& $c \in C_1^{(2,q)}$ & $u_0=\frac{(q-1)(q-2)}{2}$, $u_1=3q-3$, $u_2=\frac{q^2-3q+8}{2}$, $u_3=q-1$, $|K|=\frac{q^2+3q-2}{2}$ \\ \hline

$d=q-2$ & $c=0$ & $u_0=\frac{(q-1)^2}{2}$, $u_1=\frac{3(q-1)}{2}$, $u_2=\frac{q^2+5}{2}$, $u_3=\frac{q-1}{2}$, $|K|=\frac{q^2+2q-1}{2}$ \\ \cline{2-3}
$q \equiv 3 \bmod{4}$ & $c \in C_0^{(2,q)}$ & $u_0=\frac{q^2-3q}{2}$, $u_1=3q-1$, $u_2=\frac{q^2-3q+8}{2}$, $u_3=q-3$, $u_4=1$, $|K|=\frac{q^2+3q}{2}$ \\ \cline{2-3}
& $c \in C_1^{(2,q)}$ & $u_0=\frac{q^2-3q+4}{2}$, $u_1=3q-6$, $u_2=\frac{q^2-3q+14}{2}$, $u_3=q-2$, $|K|=\frac{q^2+3q-4}{2}$ \\ \hline

\multirow{3}{*}{$d=q-1$} & $c=0$ & $u_0=q-1$, $u_1=q^2-2q$, $u_2=2q+1$, $u_q=1$, $|K|=q^2-q+1$ \\ \cline{2-3}
                         & \multirow{2}{*}{$c \ne 0$} & $u_0=2q-4$, $u_1=q^2-4q+6$, $u_2=3q-3$, \\
                         &                            & $u_3=1$, $u_{q-1}=1$, $|K|=q^2-2q+4$ \\ \hline
\end{longtable}

If $f$ is a monomial $x^d$, then we write $DK(d,c):=DK(x^d,c)$. We list in Table~\ref{tab-interdisKakeya} the intersection distributions of several classes of dual Kakeya sets derived from monomials, and the sizes of corresponding Kakeya sets. More specifically, in Appendix~\ref{sec-appendixB}, we compute the multiplicity distribution of monomials $x^d$ over finite field $\Fq$ with $q=p^s$ and $p$ being prime, where $d \in \{p^i,p^i+1,\frac{q-1}{2},\frac{q+1}{2},q-2,q-1\}$, $0 \le i \le s-1$. As a consequence, the intersection distribution in Table~\ref{tab-interdisKakeya} follows from Remark~\ref{rem-polyint}(2) and Propositions~\ref{prop-intdisKconnection}. Given a finite field $\Fq$ and a positive integer $N$, we use $C_0^{(N,q)}$ to denote the set consisting of nonzero $N$-th powers in $\Fq$. For a positive integer $i$, we use $l_2(i)$ to denote the largest nonnegative integer, such that $2^{l_2(i)} \mid i$. We also define $l_2(0)=+\infty$.

In \cite[Table 1]{DM2}, an exhaustive computer search presented all attainable sizes of Kakeya sets in $PG(2,q)$ with $q \le 9$. We are going to show that many of these attainable sizes can be achieved by the dual Kakeya sets $DK(d,c)$ derived from monomials. In Table~\ref{tab-monoKakeya}, we list all sizes of Kakeya sets in $PG(2,q)$ with $q \le 9$. In the second column, an entry $(k,D)$ means a Kakeya set of size $k$ can be derived from $DK(d,c)$, for each $d \in D$ and some properly chosen $c$. If $D$ is an empty set, that means for each $1 \le d \le q-1$ and $c \in \Fq$, the dual Kakeya set $DK(d,c)$ never produces a Kakeya set of size $k$.

\begin{table}[H]
\begin{center}
\caption{All attainable sizes of Kakeya sets in $PG(2,q)$ and those realizable by dual Kakeya set $DK(d,c)$, $q \le 9$}
\label{tab-monoKakeya}
\begin{tabular}{|c|l|}
\hline
$q$ & attainable size using $DK(d,c)$ \\ \hline
$2$  &  $(3,\{1\})$, $(4,\{1\})$   \\ \hline
$3$  &  $(7,\{1,2\})$, $(9,\{1\})$   \\ \hline
$4$  &  $(10,\{2\})$, $(12,\{2,3\})$, $(13,\{1,3\})$, $(16,\{1\})$   \\ \hline
$5$  &  $(17,\{2,3\})$, $(18,\es)$, $(19,\{3,4\})$, $(21,\{1,4\})$, $(25,\{1\})$   \\ \hline
\multirow{2}{*}{$7$}  &  $(31,\{2,5\})$, $(32,\es)$, $(33,\{5\})$, $(34,\{4\})$, $(35,\{3,5\})$, $(36,\es)$, $(37,\{3,4\})$, $(39,\{6\})$, \\
& $(43,\{1,6\})$, $(49,\{1\})$ \\ \hline
\multirow{2}{*}{$8$}  &  $(36,\{2,4,6\})$, $(40,\{2,4,6\})$, $(42,\es)$, $(43,\{3,5\})$, $(44,\es)$, $(45,\es)$, $(46,\{3,5\})$, $(47,\es)$, \\
& $(48,\es)$, $(49,\es)$, $(52,\{7\})$, $(57,\{1,7\})$, $(64,\{1\})$ \\ \hline
\multirow{2}{*}{$9$}  &  $(49,\{2,7\})$, $(51,\es)$, $(52,\es)$, $(53,\{7\})$, $(54,\{4\})$, $(55,\es)$, $(56,\{6\})$, $(57,\{3,4,5,6\})$, \\
& $(58,\es)$, $(59,\es)$, $(60,\es)$, $(61,\{5\})$, $(62,\es)$, $(63,\{3\})$, $(67,\{8\})$, $(73,\{1,8\})$, $(81,\{1\})$ \\ \hline
\end{tabular}
\end{center}
\end{table}

We note that except the case of $(q,d)=(9,6)$, all sizes of Kakeya sets derived from $DK(d,c)$ in Table~\ref{tab-monoKakeya} can be explained by the result in Table~\ref{tab-interdisKakeya}. Moreover, we can see that by simply using monomials, the polynomial approach involving the dual Kakeya set $DK(f,c)$ already covers many attainable sizes of Kakeya sets. This justifies that the polynomial viewpoint is an effective way to study Kakeya sets.

\section{Related work and some open problems}\label{sec-5}

In this section, we mention some work related to the concepts of intersection distribution and non-hitting index. As a conclusion, we think these new concepts are interesting and deserve further investigation.

In Section~\ref{sec-1}, some $(q+1)$-sets have been characterized by their non-hitting index. Moreover, some work related to Kakeya sets essentially characterized certain $(q+2)$-set with an internal nucleus by its non-hitting index \cite{BDMS,BM,DM2}. A natural question is, can we do more along this direction? In particular, when a $(q+1)$-set or $(q+2)$-set has no internal nucleus, very little is known.

\flushleft{{\bf Open Problem 1:} Find more $(q+1)$ and $(q+2)$-sets which can be characterized by their non-hitting index. When the non-hitting index is far away from the lower and upper bounds, knowing the non-hitting index alone is not sufficient to determine the point set. In this case, we may ask what is the additional information about the intersection distribution, which is sufficient to characterize a point set.}

\flushleft{{\bf Open Problem 2:} When $q$ is odd, for a $(q+1)$-set $S$, completely determine the second largest value of $u_0(S)$. A similar problem is the determination of an upper bound on $u_0(T)$, where $T$ is a $(q+2)$-set without internal nuclei.}

\vspace{5pt}

Note that when $u_0(S_f)$ achieves the upper bound $\frac{q(q-1)}{2}$, the corresponding polynomial $f$ satisfies some favorable property, such as being an o-polynomial. It is interesting to determine the polynomial $f$ associated with a $(q+1)$-set $S$, where $u_0(S)$ is close to the upper bound.

\flushleft{{\bf Open Problem 3:} Find polynomial representations for the $(q+1)$-sets in Examples~\ref{exam-evenlarge},~\ref{exam-oddlarge}. Find proper polynomials $f$ and field elements $c$ generating dual Kakeya sets $DK(f,c)$, which lead to Kakeya sets in Table~\ref{tab-monoKakeya} and cannot be realized by monomials.

\vspace{5pt}

Given a polynomial $f$, recall that $M_i(f,c)$ equals the number of elements occurring $i$ times in the multiset $\{f(x)-cx \mid x \in \Fq\}$. Since $v_i(f)=\sum_{c \in \Fq} M_i(f,c)$, the intersection distribution and the non-hitting index of $f$ reflect the collective behaviour of the $q$ value sets $V_{f,c}=\{f(x)+cx \mid x \in \Fq\}$, where $c \in \Fq$. In this regard, there has been some closely related research.
\begin{itemize}
\item[(1)] Recall that $N_f=\{c \in \Fq \mid |V_{f,c}|=q\}$. Motivated by the investigation of complete mappings and Latin squares, Evans, Greene and Niederreiter initiated the study concerning the size of $N_f$ \cite{EGN}. Deep algebraic tools have been used in subsequent papers \cite{Tur,WMS}. Moreover, some more involved invariants other than the degree were used to bound the size of value set $V_{f,c}$ \cite{Tur,WSC}. We mention that complete mappings defined over finite fields are also called complete permutation polynomials.
\item[(2)] Counting the size of $N_f$ is equivalent to counting the number of directions determined by the graph of $f$, see for instance \cite{BBBSS}. It turns out the size of $N_f$ not only falls in a collection of disjoint intervals, but also gives information about $|V_{f,c}|$ where $c \notin N_f$, see for instance \cite[Theorem 1.1]{BBBSS}.
\item[(3)] In order to determine the intersection distribution of a polynomial $f$, we need much more information than the size of $N_f$. Recently, following some powerful construction frameworks of permutation polynomials over finite fields \cite{AGW,CK,YD,YD2,Z}, the set $N_f$ has been determined for quite a few families of polynomials $f$. Instead of supplying an exhaustive list of references concerning such polynomials, we refer to two recent excellent surveys \cite{Hou}, \cite[Section 5]{LZ}. Given a polynomial $f$, with the information about $N_f$, one can proceed to study the value set $V_{f,c}$ where $c \notin N_f$, which may lead to the intersection distribution of $f$.
\end{itemize}
For the next two open problems, the ideas in the above references could be helpful.

\flushleft{{\bf Open Problem 4:} Employ more advanced techniques to improve the lower and upper bounds on the non-hitting index of a polynomial in Proposition~\ref{prop-polybound}.}

\vspace{5pt}

\flushleft{{\bf Open Problem 5:} Determine the non-hitting index and the intersection distribution of other classes of polynomials over finite field. The non-hitting index of monomials in Table~\ref{tab-monononhitting}, which do not have a theoretical explanation so far, could be a good place to start.}

\vspace{5pt}

Suppose $f$ is a so-called Dembowski-Ostrom polynomial over $\Fq$. It has been shown in \cite[Theorem 2.3]{WZ} that $f$ is planar if and only if $f$ is $2$-to-$1$. Note that $f$ is $2$-to-$1$ if and only if $M_2(f,0)=\frac{q}{2}$. This surprising result indicates that even partial information of the multiplicity distribution of a polynomial may imply certain strong properties. More generally, we are interested in the connection between the multiplicity distribution and other properties of polynomials. In addition, the projective equivalence raised in Definition~\ref{def-projequiv} seems to be a very natural one. As explained in Remark~\ref{rem-measureconnection}, the projective equivalence offers a new geometrical angle to distinguish polynomials over finite fields. Since the intersection distribution is an invariant of the projective equivalence, one question is, to what extent, the intersection distribution determines the projective equivalence.

%

\flushleft{{\bf Open Problem 6:} Explore the connection between the multiplicity distribution and other properties of polynomials over finite fields. }

\vspace{5pt}

\flushleft{{\bf Open Problem 7:} Explore the relation between the intersection distribution and the projective equivalence. For instance, do there exist projectively inequivalent polynomials having the same intersection distribution?}

\vspace{5pt}

We mention an interesting work due to Coulter and Senger \cite{CS}. Given two sets $A$ and $B$, the authors considered a function $f:A \rightarrow B$, and defined $N_2(f)$ to be the number of pairs $(x,y) \in A \times A$, such that $x \ne y$ and $f(x)=f(y)$. Some bounds on the size of the value set $\{f(x) \mid x \in A \}$, which involves the number $N_2(f)$, were derived in \cite{CS}. A very recent work due to Ding and Tang \cite{DT} employs polynomial $f$ over finite field to construct combinatorial $t$-designs. As pointed out in \cite{DT}, determining the parameters of the $t$-design associated with a polynomial $f$ is difficult in general. On the other hand, we remark that the multiplicity distribution of $f$ implies the parameters of the associated $t$-design. Therefore, this fresh design-theoretic application offers one more motivation to study the multiplicity distribution of polynomials over finite fields. Finally, some recent progress on the intersection and multiplicity distributions of degree three polynomials has been presented in \cite{KLP}.

\section*{Appendix}

\appendix

\section{Proofs of results in Section~\ref{sec-2}}\label{sec-appendixA}

In this appendix, we present proofs of several results in Section~\ref{sec-2}. The following is the proof of Lemma~\ref{lem-upperboundstructure}.

\begin{proof}[Proof of Lemma~\ref{lem-upperboundstructure}]
Suppose $S=\{P_i \mid 0 \le i \le q\}$. Without loss of generality, we can assume that $A=\{P_i \mid 0 \le i \le k-1\}$ and $B=\{P_{k-1+i} \mid 0 \le i \le l\} \sm \{P_{k-1}\}$, where $0 \le l \le q+1-k$. For $0 \le j \le q$, define $S_j=\{ P_i \mid 0 \le i \le j \}$ to be the subset of $S$ consisting of the first $j+1$ points. In order to compute the number of lines intersecting $S$, we do the calculation in a recursive manner by computing the number of lines intersecting $S_j$, where $j$ ranges from $0$ to $q$.

For $1 \le j \le q$, there are at least $q+1-j$ external lines to $S_{j-1}$ through $P_j$. For $1 \le j \le q$, let $m_j$ be a nonnegative integer, such that there being $q+1-j+m_j$ external lines to $S_{j-1}$ through $P_j$. Clearly, $q+1-j+m_j= \sum_{l=1}^{q+1} u_l(S_j)-\sum_{l=1}^{q+1} u_l(S_{j-1})$. This means that adding $P_j$ to the set $\{P_i \mid 0 \le i \le j-1\}$ introduces $q+1-j+m_j$ lines intersecting $S_j$ and not intersecting $S_{j-1}$. Note that $m_j=0$ if and only if there are $j$ tangent lines to $S_{j-1}$ through $P_j$, and $m_j>0$ if and only if there exists at least one secant line to $S_{j-1}$ through $P_j$. For the sake of simplicity, we also define $m_0=0$. Thus, we can see that $\sum_{l=1}^{q+1} u_l(S)=\sum_{j=0}^{q} (q+1-j+m_j)=\frac{(q+1)(q+2)}{2}+\sum_{j=0}^{q} m_j$. Since $A$ is an $S$-maximal $k$-arc, we have $m_j=0$ for each $0 \le j \le k-1$ and $u_0(S)=\frac{q(q-1)}{2}-\sum_{j=k}^q m_j$.

(1) The proof is analogous to that of \cite[Lemma 1.1]{BB}. Suppose for each $P_j \in S \sm A$, there are exactly $\mu_j$ tangent lines to $A$ through $P_j$, where $\mu_j \le \la$. Note that there are $k-\mu_j$ points of $A$ not on these $\mu_j$ tangent lines. Thus, there are at most $\frac{k-\mu_j}{2}$ secant lines to $A$ through $P_j$. Thus, there are at most $\mu_j+\frac{k-\mu_j}{2}=\frac{k+\mu_j}{2}\le\frac{k+\la}{2}$ lines through $P_j$ intersecting $A$, which implies $m_j \ge k-\frac{k+\la}{2}=\frac{k-\la}{2}$, where $j \ge k$. Consequently, we have $u_0(S) \le \frac{q(q-1)}{2}-(q+1-k)\frac{k-\la}{2}$.

(2) The proof is analogous to that of \cite[Lemma 3.2]{BDMS}. For each $k \le j \le k-1+l$, we have $P_j \in B$ and by the definition of $B$, there exists a $2$-secant line to $A$ through $P_j$, which implies $m_j \ge 1$. For each $k+l \le j \le q$, we have $P_j \in S \sm (A \cup B)$. Then,  either there exists $P_l \in B$, such that $P_j$ is on the unique $2$-secant line to $A$ through $P_l$, or there are at least two secant lines to $A$ through $P_j$. In both cases, we have $m_j \ge 2$. Hence,
\begin{align*}
u_0(S)=\frac{q(q-1)}{2}-\sum_{j=k}^q m_j &=\frac{q(q-1)}{2}-\sum_{j=k}^{k-1+l} m_j-\sum_{j=k+l}^q m_j \\
                                         &\le \frac{q(q-1)}{2}-l-2(q+1-k-l)=\frac{q(q-1)}{2}-2(q+1)+2k+l.
\end{align*}
\end{proof}

The next is the proof of Lemma~\ref{lem-numtangent}.

\begin{proof}[Proof of Lemma~\ref{lem-numtangent}]
(1) Since $A$ is an $S$-maximal arc, through each point $P \in S \sm A$, there exists at least one 2-secant line to $A$. Thus, there are at most $k-2$ tangent lines to $A$ through $P$.

(2) This is a direct consequence of \cite[Theorem 10.1]{Hirs}, see also the proof of \cite[Theorem 1.2]{BB}.

(3) This is a direct consequence of \cite[Theorem 10.4]{Hirs}, see also the proof of \cite[Theorem 1.3]{BB}.

(4) Each point $P \in B$ corresponds to a pair of points $Q_1, Q_2 \in A$, such that $P \in \ol{Q_1Q_2}$. Suppose through each point of $A$, there exists at most one $2$-secant line to $A$, which contains one point of $B$. Then the $l$ points in $B$ correspond to $2l$ distinct points of $A$. Consequently, we have $2l \le k$.
\end{proof}

For $q \in \{2,4,8\}$, the following lemma provides a detailed treatment concerning $(q+1)$-sets $S$ satisfying $u_0(S) \in \{\frac{q(q-1)}{2}-\frac{q}{2},\frac{q(q-1)}{2}-\frac{q}{2}+1\}$, which is interesting in view of Theorem~\ref{thm-evenupper}.

\begin{lemma}\label{lem-evensmall}
Let $q \in \{2,4,8\}$. Let $S$ be a $(q+1)$-set and which is not a $(q+1)$-arc. Then we have the following.
\begin{itemize}
\item[(1)] If $u_0(S)=\frac{q(q-1)}{2}-\frac{q}{2}+1$, then one of the following holds.
  \begin{itemize}
    \item[(1a)] When $q=4$, then $S$ is of the form in Example~\ref{exam-evenlarge}(1).
    \item[(1b)] When $q=8$, we have $u_0(S)=25$ if either $S$ is of the form in Example~\ref{exam-evenlarge}(1), or $S$ contains an $S$-maximal $6$-arc $A$, and $S \sm A$ is a pro-arc $3$-set with respect to $S$ and $A$, containing three noncollinear points, and the three lines determined by the points of $S \sm A$ are all external lines to $A$.
  \end{itemize}
\item[(2)]  If $u_0(S)=\frac{q(q-1)}{2}-\frac{q}{2}$, then one of the following holds.
    \begin{itemize}
    \item[(2a)] $S$ is of the form in Example~\ref{exam-evenlarge}(2).
    \item[(2b)] When $q=8$, we have $u_0(S)=24$ if one of the following holds:
        \begin{itemize}
            \item[$\bullet$] $S$ contains an $S$-maximal $7$-arc $A$, where $S \sm A=\{P_1,P_2\}$. For each $1 \le i \le 2$, there are exactly two $2$-secant lines to $A$ through $P_i$, and $\ol{P_1P_2}$ is an external line to $A$.
           \item[$\bullet$] $S$ contains an $S$-maximal $6$-arc $A$, and $S \sm A$ is a pro-arc $3$-set with respect to $S$ and $A$, containing three noncollinear points, and one of the three lines determined by the points of $S \sm A$ is a tangent line to $A$ and two of the three lines are external lines to $A$.
           \item[$\bullet$] $S$ contains an $S$-maximal $6$-arc $A$, and $S \sm A$ is a pro-arc $3$-set with respect to $S$ and $A$, containing three collinear points, and the unique line determined by the points of $S \sm A$ is an external line to $A$.
           \item[$\bullet$] $S$ contains an $S$-maximal $6$-arc $A$, where $S \sm A=\{P_1,P_2,P_3\}$ and these three points are not collinear. The set $\{P_1,P_2\}$ is a pro-arc $2$-set with respect to $S$ and $A$. The line $\ol{P_1P_3}$ is the unique $2$-secant line through $P_1$ to $A$ and there exists no other $2$-secant line to $A$ through $P_3$. The three lines $\{\ol{P_1P_2},\ol{P_1P_3},\ol{P_2P_3}\}$ are all external lines to $A$.
        \end{itemize}
    \end{itemize}
\end{itemize}
\end{lemma}
\begin{proof}
The case $q \in \{2,4\}$ is easy to see. Below, we focus on the case $q=8$.

Suppose $S$ contains an $S$-maximal $k$-arc $A$. Let $B$ be a pro-arc $l$-set with respect to $S$ and $A$. In order to ensure $u_0(S) \in \{24,25\}$, by Proposition~\ref{prop-upperbound}, we can assume that $k<9$, namely $2 \le k \le 8$. If $k=8$, Example~\ref{exam-evenlarge}(1)(2) leads to (2a). For $2 \le k \le 7$, Lemma~\ref{lem-upperboundmaxarc}(1) implies $5 \le k \le 7$.

{\bf We first consider the case $k=7$}, where $S \sm A=\{P_1,P_2\}$. By Lemma~\ref{lem-upperboundmaxarc}(1), we have $u_0(S) \le 24$. By Lemma~\ref{lem-numtangent}(2), there are at most three tangent lines to $A$ through $P_i$, for $i \in \{1,2\}$. Note that the number of external lines to $A$ through $P_i$ is $4$, if there are exactly three tangent lines to $A$ through $P_i$, and the number of external lines to $A$ through $P_i$ is at least $5$, if there are at most two tangent lines to $A$ through $P_i$. We claim that $u_0(S)=24$ implies that there are exactly three tangent lines to $A$ through $P_i$, for $i \in \{1,2\}$, Otherwise, for instance, there are at least $4$ external lines to $A$ through $P_1$ and at least $5$ external lines to $A$ through $P_2$. In total, there are at least $8$ \emph{distinct} external lines to $A$ through $P_1$ or $P_2$. Consequently, the number of lines intersecting $S$ is at least $21+21+8=50$, where we have $21$ secant and $21$ tangent lines to $A$, which forces $u_0(S) \le 73-50=23$. Thus, we have shown that that there are exactly three tangent lines to $A$ through $P_i$, or equivalently, two $2$-secant lines to $A$ through $P_i$, for $i \in \{1,2\}$. If $\ol{P_1P_2}$ is a $2$-secant line to $A$, then there are one $4$-secant line and two $3$-secant lines to $S$, which by Proposition~\ref{prop-upperbound} implies $u_0(S) \le 23$. Similarly, $\ol{P_1P_2}$ cannot be a tangent line to $A$. Hence $\ol{P_1P_2}$ must be an external line to $A$. In this case, $S$ has degree $3$ and there are exactly four $3$-secant lines to $S$, which gives the first case in (2b).

{\bf Second, we consider the case $k=6$}. By Lemma~\ref{lem-upperboundstructure}(2), $u_0(S)=25$ only if $l=3$ and $u_0(S)=24$ only if $l \in \{2,3\}$. Now we consider the following cases.

Case A): $l=3$. We further split into two subcases.

Case A1): the three points of $B$ are not collinear. Then there are three lines determined by $B$. By the definition of a pro-arc set, these three lines cannot be $2$-secant lines to $A$ and we assume there are $z$ lines among them being tangent lines to $A$, where $0 \le z \le 3$. Therefore, we have
\begin{align*}
  \mbox{number of $2$-secant lines to $A$:\quad} & 15 \\
  \mbox{number of tangent lines to $A$:\quad} & 24 \\
  \mbox{number of tangent lines to $B$, which do not intersect $A$:\quad} & 2\cdot3+2z=6+2z \\
  \mbox{number of $2$-secant lines to $B$, which do not intersect $A$:\quad} & 3-z
\end{align*}
In total, the number of lines intersecting $S$ is $48+z$ and $u_0(S)=25-z$. When $z=0$ or $z=1$, we have (1b) or the second case in (2b).

Case A2): the three points of $B$ are collinear. Then there is a unique line determined by $B$, which is denoted by $\ell$. By the definition of a pro-arc set, $\ell$ cannot be a $2$-secant line to $A$. If $\ell$ is a tangent or external line to $A$, using similar argument as in Case A1), we have $u_0(S)=22$ or $u_0(S)=24$, where the latter leads to the third case in (2b).

Case B): $l=2$. In this case, let $S \sm A=\{P_1,P_2,P_3\}$ and $B=\{P_1,P_2\}$. For $i \in \{1,2\}$, denote the unique $2$-secant line to $A$ through $P_i$ by $\ell_i$. We further split into two subcases.

Case B1): the three points $P_1$, $P_2$, $P_3$ are not collinear. Let $x$ be the number of $2$-secant lines to $A$ through $P_3$, which belong to the set $\{\ell_1, \ell_2\}$. Let $y$ be the number of $2$-secant lines to $A$ through $P_3$, which are not in $\{\ell_1,\ell_2\}$ and $z$ the number of tangent lines to $A$ in the set $\{\ol{P_1P_2},\ol{P_1P_3},\ol{P_2P_3}\}$. Therefore, we have
\begin{align*}
  \mbox{number of $2$-secant lines to $A$:\quad} & 15 \\
  \mbox{number of tangent lines to $A$:\quad} & 24 \\
  \mbox{number of tangent lines to $B$, which do not intersect $A$:\quad} & 2\cdot2+7-[(x+y)+6-2(x+y)]+2(x+y+z) \\
                                                                          =&5+3x+3y+2z \\
  \mbox{number of $2$-secant lines to $B$, which do not intersect $A$:\quad} & 3-(x+y+z)
\end{align*}
In total, the number of lines intersecting $S$ is $47+2x+2y+z$ and $u_0(S)=26-2x-2y-z$. By the definition of a pro-arc set, if $x=0$, then $y \ge 2$. To ensure $u_0(S) \ge 24$, we must have $x=1$ and $y=z=0$, which leads to the fourth case of (2b).

Case B2): the three points $P_1$, $P_2$, $P_3$ are collinear. Let $y$ be the number of $2$-secant lines to $A$ through $P_3$. Note that $P_3$ is not in the pro-arc set $B$ and $P_3$ is not on the $2$-secant lines $\ell_1$ and $\ell_2$. Thus, by the definition of a pro-arc set, $y \ge 2$. Moreover, $\ol{P_1P_2P_3}$ cannot be a $2$-secant line to $A$. Using similar argument as in Case B1), we see that $u_0(S) \le 23$, whenever $\ol{P_1P_2P_3}$ is a tangent or external line to $A$.

{\bf Finally, we consider the case $k=5$}, where $S \sm A=\{P_1,P_2,P_3,P_4\}$. By Lemma~\ref{lem-upperboundmaxarc}(1), we have $u_0(S) \le 24$, where by Lemma~\ref{lem-upperboundstructure}(2) the equality holds only if $l=4$, namely $\{P_1,P_2,P_3,P_4\}$ is a pro-arc $4$-set with respect to $S$ and $A$. Suppose through each $Q \in A$, there exists at most one $2$-secant line to $A$, which contains one point of $S \sm A$. However, by Lemma~\ref{lem-numtangent}(4), we must have $l \le 2$. Hence, there must be a point $Q \in A$, so that two $2$-secant lines to $A$ through $Q$, intersecting $S \sm A$ in two different points, say, $P_1$ and $P_2$. Hence, $(A \sm \{Q\}) \cup \{P_1,P_2\}$ is a $6$-arc in $S$. Therefore, we go back to the $k \ge 6$ cases discussed before.
\end{proof}

Now we proceed to prove Theorem~\ref{thm-evenupper}.

\begin{proof}[Proof of Theorem~\ref{thm-evenupper}]
Suppose $S$ contains an $S$-maximal $k$-arc $A$ with $2 \le k \le q$. Let $B$ be a pro-arc $l$-set with respect to $S$ and $A$. Applying Lemma~\ref{lem-upperboundmaxarc}(1), we have
\begin{equation}\label{eqn-quad}
u_0(S) \le \begin{cases}
  \frac{q(q-1)}{2}-\frac{q}{2}+1 & \mbox{if $k \in \{\frac{q}{2}+2,q\}$,} \\
  \frac{q(q-1)}{2}-\frac{q}{2} & \mbox{if $k=\frac{q}{2}+1$,} \\
  \frac{q(q-1)}{2}-\frac{q}{2}-1 & \mbox{if $k \in [2,q] \sm \{\frac{q}{2}+1,\frac{q}{2}+2,q\}$,}
\end{cases}
\end{equation}
which implies $u_0(S) \le \frac{q(q-1)}{2}-\frac{q}{2}+1$.

According to \eqref{eqn-quad}, we only need to consider the three cases $k \in \{\frac{q}{2}+1,\frac{q}{2}+2,q\}$. First, if $k=q$, Example~\ref{exam-evenlarge}(1)(2) leads to (1) and (2). Second, if $k=\frac{q}{2}+2$, to ensure that $u_0(S) \in \{\frac{q(q-1)}{2}-\frac{q}{2},\frac{q(q-1)}{2}-\frac{q}{2}+1\}$, by Lemma~\ref{lem-upperboundstructure}(2) we must have $l \ge \frac{q}{2}-2$. If there exists a point $Q \in A$ and two points $P_1,P_2 \in B$, such that $\ol{P_1Q}$ and $\ol{P_2Q}$ are both $2$-secant lines to $A$, then $(A \sm \{Q\}) \cup \{P_1,P_2\}$ is a $(\frac{q}{2}+3)$-arc in $S$. Thus, there exists an $S$-maximal $k^{\pr}$-arc in $S$ with $k^{\pr} \ge \frac{q}{2}+3$. If $k^{\pr}=q$, then we go back to the first case. If $\frac{q}{2}+3 \le k^{\pr} < q$, \eqref{eqn-quad} implies $u_0(S) \le \frac{q(q-1)}{2}-\frac{q}{2}-1$. Consequently, for each $Q \in A$, there exists at most one $2$-secant line through $Q$, which contains a point of $B$. By Lemma~\ref{lem-numtangent}(4), we have $2l \le \frac{q}{2}+2$. For $l \ge \frac{q}{2}-2$, this means $q \le 12$. Third, if $k=\frac{q}{2}+1$, employing a similar approach, we can show that $S$ contains an $S$-maximal $k^{\pr}$-arc with $k^{\pr} \ge \frac{q}{2}+2$, unless $q \le 2$. Hence, either we reduce to the two cases discussed before, or we have $q \le 2$. To sum up, when $k<q$, we have $u_0(S) \in \{\frac{q(q-1)}{2}-\frac{q}{2},\frac{q(q-1)}{2}-\frac{q}{2}+1\}$ only if $q \in \{2,4,8\}$. Applying Lemma~\ref{lem-evensmall} gives the proof of (3).
\end{proof}

Next, we give a proof of Theorem~\ref{thm-oddupper}.

\begin{proof}[Proof of Theorem~\ref{thm-oddupper}]
Suppose $S$ contains an $S$-maximal $k$-arc with $2 \le k \le q$. If $\frac{2q+6}{3} \le k \le q$, then by Lemma~\ref{lem-upperboundmaxarc}(2) we have $u_0(S) \le \frac{q(q-1)}{2}-\frac{3}{2}(q+1-k)(k-\frac{2q+4}{3})$, which implies
\begin{equation}\label{eqn-upperbound1}
u_0(S) \le \begin{cases}
  \frac{q(q-1)}{2}-\frac{q-3}{3} & \mbox{if $q \equiv 0\bmod3$ and $q \ge 9$, then the equality holds only if $k=\frac{2q+6}{3}$,} \\
  \frac{q(q-1)}{2}-\frac{q-3}{2} & \mbox{if $q \equiv 1\bmod3$ and $q \ge 7$, then the equality holds only if $k \in \{\frac{2q+7}{3},q\}$,} \\
  \frac{q(q-1)}{2}-\frac{q-3}{2} & \mbox{if $q \equiv 2\bmod3$ and $q \ge 11$, then the equality holds only if $k=q$.}
\end{cases}
\end{equation}
If $2 \le k<\frac{2q+6}{3} \le q$, then by Lemma~\ref{lem-upperboundmaxarc}(2) we have $u_0(S) \le \frac{q(q-1)}{2}-(q+1-k)$, which implies
\begin{equation}\label{eqn-upperbound2}
u_0(S) \le \begin{cases}
  \frac{q(q-1)}{2}-\frac{q}{3} & \mbox{if $q \equiv 0\bmod3$, then the equality holds only if $k=\frac{2q+3}{3}$,} \\
  \frac{q(q-1)}{2}-\frac{q-1}{3} & \mbox{if $q \equiv 1\bmod3$, then the equality holds only if $k=\frac{2q+4}{3}$,} \\
  \frac{q(q-1)}{2}-\frac{q-2}{3} & \mbox{if $q \equiv 2\bmod3$, then the equality holds only if $k=\frac{2q+5}{3}$.}
\end{cases}
\end{equation}
Combining \eqref{eqn-upperbound1} and \eqref{eqn-upperbound2}, we derive the upper bound \eqref{eqn-oddupperbound}. The case $q \in \{3,5,7\}$ follows from Lemma~\ref{lem-oddsmall} below.
\end{proof}

\begin{lemma}\label{lem-oddsmall}
\quad
\begin{itemize}
    \item[(1)] If $q=3$, then $u_0(S) \le 2$, and the equality holds if and only if $S$ is of the form in Example~\ref{exam-oddlarge}(2).
    \item[(2)] If $q=5$, then $u_0(S) \le 9$, and the equality holds if and only if $S$ is of the form in Example~\ref{exam-oddlarge}(1).
    \item[(3)] If $q=7$, then $u_0(S) \le 19$, and the equality holds if and only if one of the following holds:
           \begin{itemize}
             \item[$\bullet$] $S$ is of the form in Example~\ref{exam-oddlarge}(1),
             \item[$\bullet$] $S$ contains an $S$-maximal $6$-arc, and $S \sm A$ is a pro-arc $2$-set with respect to $S$ and $A$, where the unique line determined by $S \sm A$ is an external line to $A$.
           \end{itemize}
\end{itemize}
\end{lemma}
\begin{proof}
When $q \in \{3,5\}$, the $(q+1)$-set achieving the bound \eqref{eqn-upperbound2} contains a $q$-arc, which has been determined in Example~\ref{exam-oddlarge}. When $q=7$, the two bounds \eqref{eqn-upperbound1} and \eqref{eqn-upperbound2} coincide, in which the $8$-set $S$ contains either an $S$-maximal $6$-arc or an $S$-maximal $7$-arc. The latter case follows from Example~\ref{exam-oddlarge}(1), and by using a similar approach as in the proof Lemma~\ref{lem-evensmall}, we complete the former case.
\end{proof}

Now we proceed to prove Proposition~\ref{prop-improvedupper}.

\begin{proof}[Proof of Proposition~\ref{prop-improvedupper}]
For a $(q+2)$-set $T$ with an internal nucleus, by the proof of \cite[Proposition 8]{BM}, we have
$$
\begin{cases}
u_0(T) < \frac{q(q-1)}{2}-\frac{q-1}{2} & \mbox{if $T$ contains exactly one internal nucleus,} \\
u_0(T) \le \frac{q(q-1)}{2}-\frac{q-1}{2} & \mbox{if $T$ contains two internal nuclei.}
\end{cases}
$$
If $S$ contains no internal nucleus and has a nucleus $N$, then $S \cup \{N\}$ is a $(q+2)$-set with exactly one internal nucleus $N$. Therefore, $u_0(S)=u_0(S \cup \{N\})< \frac{q(q-1)}{2}-\frac{q-1}{2}$. If $S$ contains an internal nucleus $O$ and a nucleus $N$, then $T=S \cup \{N\}$ is a $(q+2)$-set with exactly two internal nuclei $O$ and $N$. Therefore, $u_0(S)=u_0(T) \le \frac{q(q-1)}{2}-\frac{q-1}{2}$. Moreover, by \cite[Proposition 8]{BM}, $u_0(T)=\frac{q(q-1)}{2}-\frac{q-1}{2}$ if and only if $T$ is a $(q+1)$-arc plus a point $Q$, so that there are exactly two tangent lines to the $(q+1)$-arc through $Q$. Denote these two tangent lines as $\ol{P_1Q}$ and $\ol{P_2Q}$, where $P_1$ and $P_2$ are two points on the $(q+1)$-arc. Clearly, $P_1$ and $P_2$ are two internal nuclei of $T$. According to \cite[Theorem 1]{BK}, when $q$ is odd, a $(q+2)$-set contains at most $2$ internal nuclei. This forces $\{P_1,P_2\}=\{O,N\}$. Consequently, $S$ must be of the form in Example~\ref{exam-oddlarge}(2).
\end{proof}

When $q$ is odd, a $(q+2)$-set $T$ contains $0$ or $1$ or $2$ internal nuclei \cite[Theorem 1]{BK}. To our best knowledge, when $T$ contains no internal nucleus, the upper bound on $u_0(T)$ is still open. This upper bound is closely related to the upper bound on $u_0(S)$, where $S$ is a $(q+1)$-set with no internal nucleus and no nucleus. Next, we give the proof of Theorem~\ref{thm-lower}.

\begin{proof}[Proof of Theorem~\ref{thm-lower}]
Since the $(q+1)$-set $S$ has degree $n$, there exists a line $\ell$ intersecting $S$ in exactly $n$ points. Let $L$ be a subset of $S$ consisting of these $n$ points. There are $1+qn$ lines intersecting $L$, and at most $(q+1-n)^2$ lines intersecting the $q+1-n$ points in $S \sm L$ and not intersecting $L$. Consequently, $\sum_{i=1}^{q+1} u_i(S) \le 1+qn+(q+1-n)^2=q^2+2q+2-n(q+2-n)$, which implies $u_0(S) \ge n(q+2-n)-(q+1)$, where $3 \le n \le q+1$. Employing this inequality, we have
\begin{equation*}
u_0(S) \ge \begin{cases}
  0 & \mbox{if $n=q+1$,} \\
  q-1 & \mbox{if $n=q$ and $q \ge 3$,} \\
  2q-4 & \mbox{if $n \in \{3,q-1\}$ and $q \ge 4$,} \\
  3q-9 & \mbox{if $4 \le n \le q-2$ and $q \ge 7$.}
\end{cases}
\end{equation*}
Clearly, $u_0(S)=0$ if and only if $n=q+1$ and $u_0(S)=q-1$ if and only if $n=q$, which give the smallest and second smallest values of $u_0(S)$. This observation, together with Lemma~\ref{prop-upperbound}, already determine all achievable value of $u_0(S)$ when $q \in \{2,3\}$ (see Remark~\ref{rem-nonhittingspec}). Thus, we only need to consider the case $q \ge 4$ and $3 \le n \le q-1$ below. We claim that $2q-4 \le u_0(S) \le 2q-3$ only if $n=q-1$. Indeed, if $q \ge 7$, when $4 \le n \le q-2$, we have $u_0(S) \ge 3q-9>2q-3$. Hence, $2q-4 \le u_0(S) \le 2q-3$ only if $n \in \{3,q-1\}$. If $n=3$, then by Proposition~\ref{prop-upperbound}, $u_0(S)=\frac{q(q-1)}{2}-u_3(S)$. Since through each point of $S$ there are at most $\frac{q}{2}$ $3$-secant lines to $S$, $u_3(S) \le \frac{q}{2}(q+1)/3=\frac{q(q+1)}{6}$. Thus, we have $u_0(S)=\frac{q(q-1)}{2}-u_3(S) \ge \frac{q(q-1)}{2}-\frac{q(q+1)}{6}>2q-3$. For the three remaining cases $(q,n) \in \{ (4,3), (5,3), (5,4) \}$, we have $u_0(S) \in \{ 4,5 \}$, $u_0(S) \in \{8,9\}$ and $u_0(S) \in \{ 6,7 \}$, respectively. Thus, the claim holds true. The cases achieving equalities in (3) are all easy to see. Finally, the intersection distributions follow immediately from the characterization of $S$.
\end{proof}

\section{The multiplicity distribution of some power mappings}\label{sec-appendixB}

Let $q=p^s$ be a power of prime $p$. In this appendix, we determine the multiplicity distribution $(M_i(x^{d},c))_{i=0}^q$, at each $c \in \Fq$, where $d \in \{p^i,p^i+1,\frac{q-1}{2},\frac{q+1}{2},q-2,q-1\}$. Note that for simplicity, whenever referring to the multiplicity distribution, we only list the $M_i(x^{d},c)$'s with nonzero values and omit all the $M_i(x^{d},c)$'s equal to zero. Indeed, the information about nonzero $M_i(x^{d},c)$'s, already forces the rest to be zero. Recall that $C_0^{(N,q)}$ is the set consisting of nonzero $N$-th powers in $\Fq$. The following three propositions are easy to see.

\begin{proposition}
Let $q=p^s$ be a power of prime $p$. Let $0 \le i \le s-1$ be an integer with $h=\gcd(i,s)$. We have
$$
\begin{cases}
M_0(x^{p^i},c)=p^s-p^{s-h}, \quad M_{p^h}(x^{p^i},c)=p^{s-h}, & \mbox{if $c \in C_0^{(p^h-1,q)}$,} \\
M_1(x^{p^i},c)=p^s, & \mbox{if $c \notin C_0^{(p^h-1,q)}$.}
\end{cases}
$$
\end{proposition}

\begin{proposition}
Let $q$ be a prime power. We have
$$
\begin{cases}
M_0(x^{q-1},0)=q-2, \quad M_1(x^{q-1},0)=1, \quad M_{q-1}(x^{q-1},0)=1, & \\
M_0(x^{q-1},c)=1, \quad M_1(x^{q-1},c)=q-2, \quad M_{2}(x^{q-1},c)=1, & \mbox{if $c \ne 0$.}
\end{cases}
$$
\end{proposition}

\begin{proposition}
Let $q$ be an even prime power. We have
$$
\begin{cases}
M_1(x^{q-2},0)=q, & \\
M_0(x^{q-2},c)=\frac{q}{2}, \quad M_2(x^{q-2},c)=\frac{q}{2}, & \mbox{if $c \ne 0$.}
\end{cases}
$$
\end{proposition}

Next, we introduce the concept of \emph{cyclotomic numbers}. When $q$ is odd, we know that $C_0^{(2,q)}$ is the set of nonzero squares in $\Fq$ and denote the set of nonsquares in $\Fq$ by $C_1^{(2,q)}$. For $0 \le i,j \le 1$, define the cyclotomic numbers of order $2$ as
$$
(i,j)_q=|(1+C_i^{(2,q)}) \cap C_j^{(2,q)}|.
$$
The cyclotomic numbers of order $2$ are well known, see for instance \cite{Sto}.

\begin{lemma}\label{lem-cyclotomy}
Let $q$ be an odd prime power. If $q \equiv 1 \bmod 4$, then we have
$$
(0,0)_q=\frac{q-5}{4}, \quad (0,1)_q=(1,0)_q=(1,1)_q=\frac{q-1}{4}.
$$
If $q \equiv 3 \bmod 4$, then we have
$$
(0,1)_q=\frac{q+1}{4}, \quad (0,0)_q=(1,0)_q=(1,1)_q=\frac{q-3}{4}.
$$
\end{lemma}

Now we proceed to determine the multiplicity distribution of $x^{q-2}$ with $q$ odd.

\begin{proposition}
Let $q$ be an odd prime power. If $q \equiv 1 \bmod 4$, then we have
$$
\begin{cases}
M_1(x^{q-2},0)=q, & \\
M_0(x^{q-2},c)=\frac{q-1}{2}, \quad M_1(x^{q-2},c)=2, \quad M_2(x^{q-2},c)=\frac{q-5}{2}, \quad M_3(x^{q-2},c)=1, & \mbox{if $c \in C_0^{(2,q)}$,} \\
M_0(x^{q-2},c)=\frac{q-1}{2}, \quad M_1(x^{q-2},c)=1, \quad M_2(x^{q-2},c)=\frac{q-1}{2}, & \mbox{if $c \in C_1^{(2,q)}$.}
\end{cases}
$$
If $q \equiv 3 \bmod 4$, then we have
$$
\begin{cases}
M_1(x^{q-2},0)=q, &\\
M_0(x^{q-2},c)=\frac{q+1}{2}, \quad M_2(x^{q-2},c)=\frac{q-3}{2}, \quad M_3(x^{q-2},c)=1, & \mbox{if $c \in C_0^{(2,q)}$,} \\
M_0(x^{q-2},c)=\frac{q-3}{2}, \quad M_1(x^{q-2},c)=3, \quad M_2(x^{q-2},c)=\frac{q-3}{2}, & \mbox{if $c \in C_1^{(2,q)}$.}
\end{cases}
$$
\end{proposition}
\begin{proof}
Since $x^{q-2}$ is a permutation polynomial, we have $M_1(x^{q-2},0)=q$. When $c \ne 0$, we need to know the number of solutions to $x^{q-2}-cx=b$, for each $b \in \Fq$. If $b=0$, it is easy to see that
\begin{equation}\label{eqn-b=0}
|\{x \in \Fq \mid x^{q-2}-cx=0\}|=\begin{cases}
  3 & \mbox{if $c \in C_0^{(2,q)}$,} \\
  1 & \mbox{if $c \in C_1^{(2,q)}$.} \\
\end{cases}
\end{equation}
If $b \ne 0$, then $x^{q-2}-cx=b$ can only have nonzero solutions in $\Fq$. Equivalently, we only need to consider nonzero solutions in $\Fq$ to the equation
$$
x^{-1}-cx=b,
$$
where $b,c \in \Fq^*$. Replacing $x$ with $\frac{b}{c}y$, the above is equivalent to $4(y+\frac{1}{2})^2=1+\frac{4c}{b^2}$, which has $0$, $1$ or $2$ solutions in $\Fq$ if and only if $1+\frac{4c}{b^2}$ belongs to $C_1^{(2,q)}$, $\{0\}$ or $C_0^{(2,q)}$. Note that the number of nonzero $b$'s such that $1+\frac{4c}{b^2} \in C_0^{(2,q)}$ or $1+\frac{4c}{b^2} \in C_1^{(2,q)}$ can be expressed using cyclotomic numbers of order two, and $1+\frac{4c}{b^2}=0$ holds if and only if $-c \in C_0^{(2,q)}$. Together with \eqref{eqn-b=0}, we have
$$
M_0(x^{q-2},c)=2(0,1)_q, \quad M_1(x^{q-2},c)=\begin{cases}
  2 & \mbox{if $q \equiv 1 \bmod 4$,} \\
  0 & \mbox{if $q \equiv 3 \bmod 4$,}
\end{cases}
\quad M_2(x^{q-2},c)=2(0,0)_q, \quad M_3(x^{q-2},c)=1,
$$
when $c \in C_0^{(2,q)}$ and
$$
M_0(x^{q-2},c)=2(1,1)_q, \quad M_1(x^{q-2},c)=\begin{cases}
  1 & \mbox{if $q \equiv 1 \bmod 4$,} \\
  3 & \mbox{if $q \equiv 3 \bmod 4$,}
\end{cases}
\quad M_2(x^{q-2},c)=2(1,0)_q,
$$
when $c \in C_1^{(2,q)}$. Applying Lemma~\ref{lem-cyclotomy} completes the proof.
\end{proof}

For $0 \le i,j \le 1$, define $C_{i,j}^q=\{ x \in \Fq^* \mid 1-x \in C_{i}^{(2,q)}, 1+x \in C_{j}^{(2,q)} \}$. For $q=p^s$, define
$$
\de_{p,s}=\begin{cases}
  1 & \mbox{if $2 \in C_0^{(2,q)}$,} \\
  0 & \mbox{if $2 \in C_1^{(2,q)}$.}
\end{cases}
$$
Note that $\de_{p,s}=1$ if $s$ is even or $p \equiv 1,7 \bmod 8$, and $\de_{p,s}=0$ if $s$ is odd and $p \equiv 3,5 \bmod 8$. The following lemma determines the size of $C_{i,j}^q$.

\begin{lemma}\label{lem-Cij}
Let $q=p^s$ be a power of prime $p$. If $q \equiv 1 \bmod 4$, then
$$
|C_{0,0}^{q}|=\frac{q-5}{4}-\de_{p,s}, \quad |C_{1,1}^{q}|=\frac{q-5}{4}+\de_{p,s}, \quad |C_{0,1}^q|=|C_{1,0}^q|=\frac{q-1}{4}.
$$
If $q \equiv 3 \bmod 4$, then
$$
|C_{0,0}^{q}|=\frac{q-3}{4}-\de_{p,s}, \quad |C_{1,1}^{q}|=\frac{q-3}{4}+\de_{p,s}, \quad |C_{0,1}^q|=|C_{1,0}^q|=\frac{q-3}{4}.
$$
\end{lemma}
\begin{proof}
We only prove the $q \equiv 1 \bmod 4$ case and the rest is similar. Note that
$$
|\{x \in \Fq^* \mid 1-x^2 \in C_1^{(2,q)}\}|=|C_{0,1}^q|+|C_{1,0}^q|=2(0,1)_q.
$$
Since $x \in C_{0,1}^q$ if and only if $-x \in C_{1,0}^q$, we have $|C_{0,1}^q|=|C_{1,0}^q|=(0,1)_q$. Note that there are $\frac{q-3}{2}$ elements in $\Fq^*$, such that $1-x \in C_0^{(2,q)}$. We have
$$
|C_{0,0}^q|+|C_{0,1}^q|+\de_{p,s}=\frac{q-3}{2},
$$
which leads to $|C_{0,0}^q|=\frac{q-3}{2}-(0,1)_q-\de_{p,s}$. Similarly, we have $|C_{1,1}^q|=\frac{q-3}{2}-(0,1)_q+\de_{p,s}$. Applying Lemma~\ref{lem-cyclotomy} completes the proof.
\end{proof}

Now we are ready to compute the multiplicity distribution of $x^{\frac{q-1}{2}}$.

\begin{proposition}
Let $q=p^s$ be a power of prime $p$. If $q \equiv 1 \bmod 4$, then
$$
\begin{cases}
M_0(x^{\frac{q-1}{2}},0)=q-3, \quad M_1(x^{\frac{q-1}{2}},0)=1, \quad M_{\frac{q-1}{2}}(x^{\frac{q-1}{2}},0)=2, & \\
M_0(x^{\frac{q-1}{2}},c)=\frac{q+3}{4}, \quad M_1(x^{\frac{q-1}{2}},c)=\frac{q-3}{2}, \quad M_{2}(x^{\frac{q-1}{2}},c)=\frac{q+3}{4}, & \mbox{if $c \ne 0$}.
\end{cases}
$$
If $q \equiv 3 \bmod 4$, then
$$
\begin{cases}
M_0(x^{\frac{q-1}{2}},0)=q-3, \quad M_1(x^{\frac{q-1}{2}},0)=1, \quad M_{\frac{q-1}{2}}(x^{\frac{q-1}{2}},0)=2, &\\
M_0(x^{\frac{q-1}{2}},c)=\frac{q+5}{4}-\de_{p,s}, \quad M_1(x^{\frac{q-1}{2}},c)=\frac{q-3}{2}+2\de_{p,s}, \quad M_2(x^{\frac{q-1}{2}},c)=\frac{q-3}{4}-\de_{p,s}, & \\
M_3(x^{\frac{q-1}{2}},c)=1, & \mbox{if $c \in C_0^{(2,q)}$}, \\
M_0(x^{\frac{q-1}{2}},c)=\frac{q-3}{4}+\de_{p,s}, \quad M_1(x^{\frac{q-1}{2}},c)=\frac{q+3}{2}-2\de_{p,s}, \quad M_2(x^{\frac{q-1}{2}},c)=\frac{q-3}{4}+\de_{p,s},& \mbox{if $c \in C_1^{(2,q)}$}.
\end{cases}
$$
\end{proposition}
\begin{proof}
We only prove the $q \equiv 3 \bmod 4$ case and the rest is similar. When $c=0$, the multiplicity distribution is clear. When $c \ne 0$, we need to determine the number of solutions in $\Fq$ to
\begin{equation}\label{eqn-(q-1)/2}
x^{\frac{q-1}{2}}-cx=b,
\end{equation}
for each $b \in \Fq$. Note that $0$ is a solution in $\Fq$ to \eqref{eqn-(q-1)/2} if and only if $b=0$. Thus, we only need to consider the nonzero solutions in $\Fq$ to \eqref{eqn-(q-1)/2}, which is equivalent to
\begin{equation}\label{eqn-SNS}
\begin{cases}
x=\frac{1-b}{c}, & \mbox{if $x \in C_0^{(2,q)}$} \\
x=-\frac{1+b}{c}, & \mbox{if $x \in C_1^{(2,q)}$}
\end{cases}
\end{equation}
We first consider the case $c \in C_0^{(2,q)}$. For $b=0$, \eqref{eqn-(q-1)/2} has three solutions in $\Fq$. For $b\in\{\pm1\}$, we can see that \eqref{eqn-(q-1)/2} has $\de_{p,s}$ solution in $\Fq$. Consequently, by \eqref{eqn-SNS}, we have
\begin{align*}
M_0(x^{\frac{q-1}{2}},c)=|C_{1,1}^q|+2(1-\de_{p,s}), & \quad M_1(x^{\frac{q-1}{2}},c)=|C_{0,1}^q|+|C_{1,0}^q|+2\de_{p,s}, \\ M_2(x^{\frac{q-1}{2}},c)=|C_{0,0}^q|, & \quad M_3(x^{\frac{q-1}{2}},c)=1.
\end{align*}
The case $c \in C_1^{(2,q)}$ is also similar: for $b=0$, \eqref{eqn-(q-1)/2} has one solution in $\Fq$, and for $b\in\{\pm1\}$, we can see that \eqref{eqn-(q-1)/2} has $1-\de_{p,s}$ solutions in $\Fq$. Consequently, by \eqref{eqn-SNS}, we have
\begin{align*}
M_0(x^{\frac{q-1}{2}},c)=|C_{0,0}^q|+2\de_{p,s}, & \quad M_1(x^{\frac{q-1}{2}},c)=|C_{0,1}^q|+|C_{1,0}^q|+1+2(1-\de_{p,s}), \quad M_2(x^{\frac{q-1}{2}},c)=|C_{1,1}^q|.
\end{align*}
Applying Lemma~\ref{lem-Cij} completes the proof.
\end{proof}

The multiplicity distribution of $x^{\frac{q+1}{2}}$ can also be determined in a similar way.

\begin{proposition}
Let $q$ be a prime power. If $q \equiv 1 \bmod 4$, then
$$
\begin{cases}
M_1(x^{\frac{q+1}{2}},c)=q, & \mbox{if $c \in \{0\} \cup C_{0,0}^q \cup C_{1,1}^q$,}\\
M_0(x^{\frac{q+1}{2}},c)=\frac{q-1}{2}, \quad M_1(x^{\frac{q+1}{2}},c)=1, \quad M_2(x^{\frac{q+1}{2}},c)=\frac{q-1}{2}, & \mbox{if $c \in C_{0,1}^q \cup C_{1,0}^q$,} \\
M_0(x^{\frac{q+1}{2}},c)=\frac{q-1}{2}, \quad M_1(x^{\frac{q+1}{2}},c)=\frac{q-1}{2}, \quad M_{\frac{q+1}{2}}(x^{\frac{q+1}{2}},c)=1, & \mbox{if $c=\pm1$.}
\end{cases}
$$
If $q \equiv 3 \bmod 4$, then
$$
\begin{cases}
M_1(x^{\frac{q+1}{2}},c)=q, & \mbox{if $c \in C_{0,1}^q \cup C_{1,0}^q$,}\\
M_0(x^{\frac{q+1}{2}},c)=\frac{q-1}{2}, \quad M_1(x^{\frac{q+1}{2}},c)=1, \quad M_2(x^{\frac{q+1}{2}},c)=\frac{q-1}{2}, & \mbox{if $c \in \{0\} \cup C_{0,0}^q \cup C_{1,1}^q$,} \\
M_0(x^{\frac{q+1}{2}},c)=\frac{q-1}{2}, \quad M_1(x^{\frac{q+1}{2}},c)=\frac{q-1}{2}, \quad M_{\frac{q+1}{2}}(x^{\frac{q+1}{2}},c)=1, & \mbox{if $c=\pm1$.}
\end{cases}
$$
\end{proposition}
\begin{proof}
We only prove the $q \equiv 1 \bmod 4$ case and the rest is similar. Clearly, $M_1(x^{\frac{q+1}{2}},0)=q$. When $c \ne 0$, we can see that $0$ is an solution in $\Fq$ to
\begin{equation}\label{eqn-(q+1)/2}
x^{\frac{q+1}{2}}-cx=b
\end{equation}
if and only if $b=0$. Thus, we need to consider the number of nonzero solutions in $\Fq$ to $\eqref{eqn-(q+1)/2}$ for each $b \in \Fq$, which is equivalent to
\begin{equation}\label{eqn-SNS2}
\begin{cases}
(1-c)x=b, & \mbox{if $x \in C_0^{(2,q)}$,} \\
-(1+c)x=b, & \mbox{if $x \in C_1^{(2,q)}$.}
\end{cases}
\end{equation}
When $c=1$, if $b=0$, there are $\frac{q+1}{2}$ solutions in $\Fq$ to \eqref{eqn-(q+1)/2}. For $b \ne 0$, \eqref{eqn-(q+1)/2} has one solution if and only if $\frac{b}{2} \in C_1^{(2,q)}$. Therefore, we obtain the multiplicity distribution of $x^{\frac{q+1}{2}}$ at $1$. A similar approach applies to the $c=-1$ case. If $c \in C_{0,0}^q \cup C_{1,1}^q$, for each $b \in \Fq^*$, exactly one of two equations in \eqref{eqn-SNS2} has one solution. If $c \in C_{0,1}^q$, each of two equations in \eqref{eqn-SNS2} has one solution in $\Fq$ if $b \in C_0^{(2,q)}$ and none of them has solution in $\Fq$ if $b \in C_1^{(2,q)}$. Analogously, if $c \in C_{1,0}^q$, each of two equations in \eqref{eqn-SNS2} has one solution in $\Fq$ if $b \in C_1^{(2,q)}$ and none of them has solution in $\Fq$ if $b \in C_0^{(2,q)}$. Consequently, we complete the multiplicity distribution of $x^{\frac{q+1}{2}}$.
\end{proof}

Finally, we compute the multiplicity distribution of $x^{p^i+1}$, which is a direct consequence of \cite[Theorem 5.6]{Blu}. Recall that for a positive integer $i$, the number $l_2(i)$ is the largest nonnegative integer such that $2^{l_2(i)} \mid i$ and $l_2(0)=+\infty$.

\begin{proposition}\label{prop-Blu}
Let $q=p^s$ be a power of prime $p$. Let $0 \le i \le s-1$ be an integer with $h=\gcd(i,s)$. If $l_2(h)<l_2(s)$, then
$$
\begin{cases}
M_0(x^{p^i+1},0)=\frac{p^h(p^s-1)}{p^h+1}, \quad M_1(x^{p^i+1},0)=1, \quad M_{p^h+1}(x^{p^i+1},0)=\frac{p^s-1}{p^h+1}, & \\
M_0(x^{p^i+1},c)=\frac{p^{s+h}-p^h}{2(p^h+1)}, \quad M_1(x^{p^i+1},c)=p^{s-h}, \quad M_2(x^{p^i+1},c)=\frac{p^{s+h}-2p^s+p^h}{2(p^h-1)}, & \\ M_{p^h+1}(x^{p^i+1},c)=\frac{p^{s-h}-p^h}{p^{2h}-1}, & \mbox{if $c \ne 0$}.
\end{cases}
$$
If $p=2$ and $l_2(h) \ge l_2(s)$, then
$$
\begin{cases}
M_1(x^{p^i+1},0)=p^s, & \\
M_0(x^{p^i+1},c)=\frac{p^{s+h}+p^h}{2(p^h+1)}, \quad M_1(x^{p^i+1},c)=p^{s-h}-1, \quad M_2(x^{p^i+1},c)=\frac{p^{s+h}-2p^s+p^h}{2(p^h-1)}, & \\ M_{p^h+1}(x^{p^i+1},c)=\frac{p^{s-h}-1}{p^{2h}-1}, & \mbox{if $c \ne 0$}.
\end{cases}
$$
If $p$ is odd and $l_2(h) \ge l_2(s)$, then
$$
\begin{cases}
M_0(x^{p^i+1},0)=\frac{p^s-1}{2}, \quad M_1(x^{p^i+1},0)=1, \quad M_{2}(x^{p^i+1},0)=\frac{p^s-1}{2}, & \\
M_0(x^{p^i+1},c)=\frac{p^{s+h}-1}{2(p^h+1)}, \quad M_1(x^{p^i+1},c)=p^{s-h}, \quad M_2(x^{p^i+1},c)=\frac{p^{s+h}-2p^s+1}{2(p^h-1)}, & \\ M_{p^h+1}(x^{p^i+1},c)=\frac{p^{s-h}-1}{p^{2h}-1}, & \mbox{if $c \ne 0$}.
\end{cases}
$$
\end{proposition}
\begin{proof}
We only prove the $l_2(h)<l_2(s)$ case and the rest is similar. When $c=0$, noting that $(p^i+1,p^s-1)=p^h+1$, we immediately get the multiplicity distribution of $x^{p^i+1}$ at $0$. When $c \ne 0$, we need to compute the number of solutions in $\Fq$ to
\begin{equation}\label{eqn-p^i+1}
x^{p^i+1}-cx-b=0.
\end{equation}
If $b=0$, then \eqref{eqn-p^i+1} has two solutions. If $b \ne 0$, replacing $x$ with $-\frac{b}{c}y$, we have
\begin{equation}\label{eqn-p^i+1new}
y^{p^i+1}+(-1)^{p^i+1}\frac{c^{p^i+1}}{b^{p^i}}y-(-1)^{p^i+1}\frac{c^{p^i+1}}{b^{p^i}}=0.
\end{equation}
For $0 \le i \le p^h+1$, let $N_i$ be the number of $b \in \Fq^*$, such that \eqref{eqn-p^i+1new} has $i$ solutions in $\Fq$. By \cite[Theorem 5.6]{Blu}, we have
$$
N_0=\frac{p^{s+h}-p^h}{2(p^h+1)}, \quad N_1=p^{s-h}, \quad N_2=\frac{(p^h-2)(p^s-1)}{2(p^h-1)}, \quad N_{p^h+1}=\frac{p^{s-h}-p^h}{p^{2h}-1}.
$$
Together with the $b=0$ case, we derive the multiplicity distribution.
\end{proof}

\section*{Acknowledgement}

Shuxing Li is supported by the Alexander von Humboldt Foundation. The authors wish to thank Keith Mellinger for supplying a copy of \cite{DM2} and the anonymous reviewers for their very detailed and helpful comments.

\end{document}